\documentclass[acmtocl, acmnow]{acmtrans2m-play-sansperms}

\usepackage{amsmath,amsfonts,amssymb,euscript, graphicx,epsfig,
enumerate,float,afterpage, subfigure, ifthen}%

\newtheorem{thm}{Theorem}
\newtheorem{cor}{Corollary}
\newtheorem{lem}{Lemma}

\newcommand{\expect}[1]{\mathbb{E}\left\{#1\right\}}
\newcommand{\defequiv}{\mbox{\raisebox{-.3ex}{$\overset{\vartriangle}{=}$}}}

\newcommand{\bv}[1]{{\boldsymbol{#1} }}
\newcommand{\script}[1]{{{\cal{#1} }}}

\title{Dynamic Product Assembly and Inventory Control for Maximum Profit}
\author{Michael J. Neely and Longbo Huang}

\begin{abstract}
We consider a manufacturing plant that purchases raw materials for product
assembly and then sells the final products to customers.  There are $M$ types of raw materials
and $K$ types of products, and each product uses a certain subset of raw materials
for assembly.  The plant operates
in slotted time, and every slot it makes decisions about re-stocking materials and
pricing the existing products in reaction to (possibly time-varying) 
material costs and consumer demands.    
We develop a dynamic purchasing and pricing policy  that yields
time average profit within $\epsilon$ of optimality, for any given $\epsilon>0$,
with a worst case storage buffer requirement that is 
$O(1/\epsilon)$.   The policy can be implemented easily for large $M$, $K$,
yields fast convergence times, and is robust to non-ergodic system dynamics. 
\end{abstract}

\category{...}{...}{...}


\begin{document}

\begin{bottomstuff} 
Michael J. Neely$\dagger$ and Longbo Huang$\star$ are with the  
Electrical Engineering Department 
at the University of Southern California, Los Angeles, 
CA \{$\dagger$web: http://www-rcf.usc.edu/$\sim$mjneely\}, \{$\star$web: http://www-scf.usc.edu/$\sim$longbohu\}).

This material is supported in part  by one or more of
the following: the DARPA IT-MANET program
grant W911NF-07-0028,  the NSF Career grant CCF-0747525.
\end{bottomstuff}

\maketitle

\section{Introduction}

This paper considers the problem of maximizing time average profit at a product
assembly plant.  The plant  manages the purchasing, 
assembly, and pricing of $M$ types of raw materials and $K$ types of 
products.  Specifically, the plant maintains a storage buffer  
for each of the $M$ materials, and can assemble each product from some specific
combination of materials.  The system operates in slotted time with 
normalized slots $t \in\{0, 1, 2, \ldots\}$.  Every slot, the plant makes decisions about
purchasing new raw materials and pricing
the $K$ products for sale to the consumer.  This is done in reaction to material
costs and consumer demand functions that are known on each slot but
can change randomly from slot to slot according to a stationary 
process with a  possibly unknown
probability distribution.

It is well known that the problem of maximizing time average profit in such a system 
can be treated using dynamic programming and  Markov decision theory.   A textbook example 
of this approach  for  a single product (single queue) problem is given in \cite{bertsekas-dp},
where inventory storage costs are also considered. 
However, such approaches may be prohibitively complex for problems with large dimension, 
as the state space grows exponentially with the number of queues. 
Further, these techniques require knowledge of the probabilities  that govern
purchasing costs and consumer demand functions.   Case studies of multi-dimensional 
inventory control are treated in \cite{ndp-inventory} using a 
lower complexity neuro-dynamic programming framework, 
which approximates the optimal value function used in traditional dynamic programming.
Such algorithms fine-tune the parameters of the approximation by  either 
offline simulations or online feedback (see also \cite{bertsekas-neural}\cite{approx-dp}). 

In this paper, we consider a different approach that does not attempt to approximate
dynamic programming.    Our  algorithm
reacts to the current system state and does not require knowledge of the probabilities that
affect future states. 
Under mild ergodicity assumptions on the material supply and consumer demand processes, 
we show that the algorithm can push time average profit to  
within $\epsilon$ of optimality,  for any arbitrarily small value
$\epsilon>0$.  This can be achieved by finite storage buffers of size
$cT_{\epsilon}/\epsilon$, where $c$ is a coefficient that is polynomial in $K$ and $M$, 
and $T_{\epsilon}$ is a constant that depends on the ``mixing time'' of the processes. 
In the special case when these processes are i.i.d. over slots, we have $T_{\epsilon} = 1$ for
all $\epsilon>0$, and so the buffers are size $O(1/\epsilon)$.\footnote{If the material supply and 
consumer demand processes are modulated by finite state ergodic Markov chains, then $T_{\epsilon} = O(\log(1/\epsilon))$ and so the buffers are size
$O((1/\epsilon)\log(1/\epsilon))$.} 
The algorithm can be implemented in real time even for problems with large dimension (i.e., large $K$ and $M$). 
 Thus, our framework circumvents the ``curse of dimensionality'' problems
associated with dynamic programming.  This is because we are not asking the same question that could be 
asked by dynamic programming approaches:  Rather than attempting to maximize profit subject to finite storage
buffers, we attempt to reach the more difficult target of pushing profit arbitrarily close to the maximum
that can be achieved in systems with \emph{infinite buffer space}.  We can approach this
optimality with finite buffers of size $O(1/\epsilon)$, although this may not be the optimal buffer size 
tradeoff (see \cite{neely-energy-delay-it}\cite{neely-utility-delay-jsac} for tradeoff-optimal algorithms
in a communication network).    A dynamic program might be able to achieve the same profit with smaller buffers, but would
contend with curse of dimensionality issues. 

Prior work on inventory control  with system models similar to our
own is found  in \cite{pricing-short-life-cycles} \cite{assemble-to-order-dp} \cite{assemble-to-order-plambeck06} 
and references therein.  Work in \cite{pricing-short-life-cycles} considers 
a single-dimensional inventory problem where a fixed number of products are  sold over 
a finite horizon with a constant but unknown customer arrival rate.  A set of coupled differential
equations are derived for the optimal policy using Markov decision theory.  Work 
in \cite{assemble-to-order-dp} provides structural results for multi-dimensional inventory
problems with product assembly, again using Markov decision theory, and obtains 
numerical results for a two-dimensional system.  A multi-dimensional product assembly 
problem is treated in \cite{assemble-to-order-plambeck06} for stochastic customer arrivals
with  fixed and known rates.  The complexity issue is treated by considering a large volume
limit and using results of heavy traffic theory.  The work in \cite{assemble-to-order-plambeck06} also
considers joint optimal price decisions, but chooses all prices at time zero and holds them 
constant for all time thereafter. 

Our analysis uses the ``drift-plus-penalty'' 
framework of stochastic network optimization 
developed for queueing networks 
in \cite{now}\cite{neely-fairness-infocom05}\cite{neely-energy-it}.  Our problem 
is most similar to the work 
in \cite{jiang-processing}, which
uses this framework to address \emph{processing networks} that queue components
that must be combined with other components.   The work in \cite{jiang-processing} 
treats multi-hop networks and maximizes throughput and throughput-utility in these
systems using a \emph{deficit max-weight} algorithm that uses ``deficit queues''  to 
keep track of the deficit created when a component cannot be processed due to a missing
part.   Our paper does not consider a multi-hop network, but has similar challenges when 
we do not have enough inventory to build a desired product.  Rather than using deficit
queues, we use a different type of Lyapunov function that avoids deficits entirely. 
Our formulation also considers the purchasing and pricing aspects of the problem, particularly for a 
manufacturing plant, and considers arbitrary (possibly non-ergodic) material supply and consumer
demand processes.

Previous work in \cite{two-prices-allerton07} uses the drift-plus-penalty 
framework in 
a related revenue maximization problem for a wireless service provider. 
In that context, a two-price result demonstrates that dynamic pricing must be used
to maximize time average profit (a single price is often not enough, although two prices
are sufficient). 
The problem in this paper can be viewed as the ``inverse'' of the service provider
problem, and has an extra constraint that requires the plant to maintain enough inventory
for a sale to take place.  However, a similar two-price structure applies here, 
so that time-varying prices are generally required for optimality, even
if material costs and consumer demands do not change with time. This is a simple phenomenon
that often arises when maximizing the expectation of a non-concave profit function subject to a 
limited supply of raw materials.  In the real world, product providers often use a regular
price that applies most of the time, with reduced ``sale'' prices that are offered
less frequently.  While the incentives for two-price behavior in the real world 
are complex and are often related to product expiration dates (which is not part of our 
mathematical model), two-price  (or multi-price) behavior can arise even in markets 
with non-perishable goods.
Time varying prices also arise  in other contexts, such as  
in the work \cite{pricing-short-life-cycles} which 
treats the sale of a fixed amount of items over a finite time horizon.  


It is important to note that the term ``dynamic pricing''  is often associated with the
practice of price discrimination between consumers with different 
demand functions.  It is well known that charging different consumers different prices is 
tantalizingly profitable (but often illegal). 
Our model does not use such price discrimination, as it offers the same price to all consumers.  
However, the revenue earned from our time-varying
strategy may be indirectly reaping benefits that are similar to those achievable by price
discrimination, without  the inherent unfairness.  This 
is because the aggregate demand function is 
composed of individual demands from consumers with different preferences, which can 
partially be exploited with a time-varying price that operates on  two different price
regions. 

The outline of this paper is as follows: 
In the next section we specify the system model.  The optimal time average profit
is characterized in Section \ref{section:optimal-profit}, where the two-price behavior is
also noted.   Our dynamic control policy is developed in Section  
\ref{section:dynamic} for an i.i.d. model of material cost and consumer demand states.
Section \ref{section:ergodic} treats a more general ergodic model, and arbitrary (possibly non-ergodic)
processes are treated in Section \ref{section:non-ergodic}.

\section{System Model} \label{section:model}

There are $M$ types of raw materials, and each is stored in a different
storage buffer at the plant.  Define $Q_m(t)$ as the (integer) 
number of type $m$ materials
in the plant on slot $t$.  
We temporarily assume all storage buffers have infinite space, and later we show that our
solution can be implemented with finite buffers of size $O(1/\epsilon)$, where the $\epsilon$ parameter
determines a profit-buffer tradeoff.

Let $\bv{Q}(t) = (Q_1(t), \ldots, Q_M(t))$ be the vector of queue sizes, also called 
 the \emph{inventory vector}.    From these materials, the plant can manufacture $K$ types
 of products. Define $\beta_{mk}$ as the (integer) number of type $m$ materials
 required for creation of a single item of product $k$ (for $m \in \{1, \ldots, M\}$ and 
 $k \in \{1, \ldots, K\}$).   We assume that products are assembled quickly, so that 
 a product requested during slot $t$ can be assembled on the same slot, provided that
 there are enough raw materials.\footnote{Algorithms that yield similar performance
 but require products to be assembled one slot before they are delivered can be designed based
 on simple modifications, briefly discussed
 in Section \ref{subsection:assembly-delay}.}  Thus, the plant must have $Q_m(t) \geq \beta_{mk}$
 for all $m \in \{1, \ldots, M\}$ in order to sell one product of type $k$ on slot $t$, and must have
 twice this amount of materials in order to sell two type $k$ products, etc.   
  The simplest
 example is when each raw material itself represents a finished product, which corresponds to the 
 case $K=M$, $\beta_{mm} = 1$ for all $m$,  $\beta_{mk} = 0$ for $m \neq k$.
 However, our model allows for more complex assembly structures, possibly
 with different products requiring some overlapping materials.

Every slot $t$, the plant must decide how many new raw materials to purchase and what
price it should charge for its products.    Let $\bv{A}(t) = (A_1(t), \ldots, A_M(t))$ represent the vector
of the (integer) 
number of  new raw materials purchased on slot $t$. 
Let $\tilde{\bv{D}}(t) = (\tilde{D}_1(t), \ldots, \tilde{D}_K(t))$ be the vector of the (integer) 
number of products sold
on slot $t$. The queueing dynamics for $m \in \{1, \ldots, M\}$ are thus: 
 \begin{equation} \label{eq:dynamics} 
 Q_m(t+1) = \max\left[Q_m(t) - \sum_{k=1}^K \beta_{mk}\tilde{D}_k(t), 0\right] + A_m(t) 
 \end{equation} 
 
 Below we describe the 
 pricing decision model that affects product sales 
 $\tilde{\bv{D}}(t)$, and the cost model associated with purchasing decisions $\bv{A}(t)$.

\subsection{Product Pricing and the Consumer Demand Functions}

For each slot $t$ and each commodity $k$, the plant must decide if it 
desires to offer commodity $k$ for sale, and, if so, what price it should charge. 
Let $Z_k(t)$ represent a binary variable 
that is $1$ if commodity $k$ is offered and is $0$ else.  Let $P_k(t)$ represent
the per-unit price for product $k$ on slot $t$.   We assume that prices $P_k(t)$ are chosen 
within a compact set $\script{P}_k$ of price options.   Thus: 
\begin{equation} \label{eq:p-k-constraint} 
P_k(t) \in \script{P}_k  \: \mbox{ for all products $k \in \{1, \ldots, K\}$ and all slots $t$} 
\end{equation} 
The sets $\script{P}_k$ 
include only non-negative prices and have a finite maximum price $P_{k, max}$.  For example, 
the set $\script{P}_k$ might represent the interval $0 \leq p \leq P_{k,max}$, 
or might represent a discrete set of prices separated by some minimum price unit.  
 Let $\bv{Z}(t) = (Z_1(t), \ldots, Z_K(t))$ 
and $\bv{P}(t) = (P_1(t), \ldots, P_K(t))$ be vectors of these decision variables.

Let $Y(t)$ represent the \emph{consumer demand state} for slot $t$, which represents
any factors that affect the expected purchasing decisions of consumers on slot $t$. 
Let $\bv{D}(t) = (D_1(t), \ldots, D_K(t))$ be the resulting \emph{demand vector}, where
$D_k(t)$ represents the (integer)
amount of type $k$ products that consumers want to buy
in reaction to the current price $P_k(t)$ and under the current demand state $Y(t)$. 
Specifically, we assume that $D_k(t)$ is a random
variable that depends on $P_k(t)$ and $Y(t)$, is conditionally i.i.d. over all slots
with the same $P_k(t)$ and $Y(t)$ values, and satisfies:  
\begin{equation} \label{eq:demand-function} 
F_k(p, y) = \expect{D_k(t) \left|\right. P_k(t) = p,  Y(t) = y}  \: \: \forall p \in \script{P}_k, y \in \script{Y}
\end{equation}
The $F_k(p,y)$ function is assumed to be continuous in $p\in \script{P}$ for 
each $y \in \script{Y}$.\footnote{This ``continuity'' is automatically satisfied in the
case when $\script{P}_k$ is a finite set of points.  Continuity of $F_k(p,y)$ 
and compactness of $\script{P}_k$
ensures that  linear
functionals of $F_k(p,y)$ have well defined maximizers $p \in \script{P}_k$.} 
We assume that the current demand state $Y(t)$ is known to  the plant at the beginning
of slot $t$, and that the 
demand function $F_k(p, y)$ is also known to the plant.  The process $Y(t)$ takes values in a 
finite or countably infinite set $\script{Y}$, and is assumed to be stationary and ergodic with steady
state probabilities $\pi(y)$, so that: 
\[ \pi(y) = Pr[Y(t) = y] \: \: \forall y \in \script{Y}, \forall t \]
The probabilities $\pi(y)$ are not necessarily known to the plant. 

We assume that the maximum demand for each product $k\in \{1, \ldots, K\}$  is deterministically
bounded by a finite 
integer $D_{k, max}$, so that regardless of price $\bv{P}(t)$ or the demand state $Y(t)$, we have: 
\[ 0 \leq D_k(t) \leq D_{k, max} \: \: \mbox{ for all slots $t$ and all products $k$} \]
This boundedness assumption is useful for analysis.  Such a finite bound is natural 
in cases when the maximum number of customers is limited on any given slot. 
The bound might also be artificially enforced by the plant due to physical constraints 
that limit the number of orders that can be fulfilled on one slot.   Define $\mu_{m,max}$ as
the resulting maximum demand for raw materials of type $m$ on a given slot: 
\begin{equation} \label{eq:mu-max} 
 \mu_{m, max} \defequiv \sum_{k=1}^K\beta_{mk} D_{k,max} 
 \end{equation}

If there is a sufficient amount of raw materials to fulfill all demands in the vector
$\bv{D}(t)$, and if $Z_k(t) = 1$ for all $k$ such that $D_k(t)>0$ (so that product $k$ is
offered for sale), 
then the number of products sold is equal to the demand vector:  $\tilde{\bv{D}}(t) = \bv{D}(t)$.
We are guaranteed to have enough inventory to meet the demands on slot $t$ if $Q_m(t) \geq \mu_{m,max}$
for all $m \in \{1, \ldots, M\}$. 
However, there may not always be enough inventory to fulfill all demands, in which 
case we require a \emph{scheduling decision} that decides how many units of each product
will be assembled to meet a subset of the demands. 
The value of $\tilde{\bv{D}}(t) = (\tilde{D}_1(t), \ldots, \tilde{D}_K(t))$
 must be chosen as an integer vector that satisfies the following \emph{scheduling constraints}: 
\begin{eqnarray}
0 \leq \tilde{D}_k(t) \leq Z_k(t)D_k(t) &  \forall k \in \{1, \ldots, K\} \label{eq:scheduling-constraints1} \\
Q_m(t) \geq \sum_{k=1}^K\beta_{mk}\tilde{D}_k(t)  &  \forall m\in\{1, \ldots, M\} \label{eq:scheduling-constraints2} \end{eqnarray}

\subsection{Raw Material Purchasing Costs} \label{subsection:cost-state}

 Let $X(t)$ represent the \emph{raw material supply state} 
on slot $t$, which contains components that  affect the purchase price of new raw materials.  
Specifically, we assume that $X(t)$ has the form: 
\[ X(t) = [(x_1(t), \ldots, x_M(t)); (s_1(t), \ldots, s_M(t))] \] 
where $x_m(t)$ is the per-unit price of raw material $m$ on slot $t$, 
and $s_m(t)$ is the maximum amount of raw material $m$ available for
sale on slot $t$. 
We assume that 
$X(t)$ takes values on some finite or countably infinite set $\script{X}$,
and that $X(t)$ is stationary and ergodic 
with probabilities: 
\[  \pi(x) = Pr[X(t) = x] \: \: \forall x \in \script{X}, \forall t \]
The $\pi(x)$ probabilities
are not necessarily known to the plant. 

Let $c(\bv{A}(t), X(t))$ be the total 
 cost incurred by the plant for purchasing a vector $\bv{A}(t)$ of new materials
under the supply state $X(t)$: 
\begin{equation} \label{eq:example-cost} 
c(\bv{A}(t), X(t)) = 
                          \sum_{m=1}^M x_m(t) A_m(t)                     
                                       \end{equation} 
   We assume that $\bv{A}(t)$ is limited by the 
                                 constraint $\bv{A}(t) \in \script{A}(X(t))$, where  $\script{A}(X(t))$ is the set of
                                 all vectors 
$\bv{A}(t) = (A_1(t), \ldots, A_M(t))$ such that for all $t$:  
\begin{eqnarray}
0 \leq A_m(t) \leq \min[A_{m,max}, s_m(t)]  & \forall m \in \{1, \ldots, M\} \label{eq:Amax} \\
A_m(t) \mbox{ is an integer } & \forall m \in \{1, \ldots, M\}  \label{eq:integer-a} \\
c(\bv{A}(t), X(t)) \leq c_{max} &  \label{eq:cmax} 
\end{eqnarray}
where $A_{m, max}$ and $c_{max}$ are finite bounds on the total amount of each raw material
that can be purchased, and the total cost of these purchases on one slot, respectively. 
  These finite bounds might arise  from the limited supply of raw materials, 
or might be artificially imposed by the plant in order to limit the risk associated with 
investing in new raw materials on any  given slot.   A simple special case is when 
there is a finite maximum price $x_{m, max}$ for raw material $m$ at any time, and when
$c_{max} = \sum_{m=1}^M x_{m, max} A_{m,max}$.  In this case, the constraint (\ref{eq:cmax}) 
is redundant.

\subsection{The Maximum Profit Objective} \label{subsection:queueing-dynamics}

 Every slot $t$, the plant observes the current queue vector $\bv{Q}(t)$, 
 the current demand state $Y(t)$, and the current supply state $X(t)$, 
 and chooses a purchase vector $\bv{A}(t) \in \script{A}(X(t))$ and 
 pricing vectors  $\bv{Z}(t)$, $\bv{P}(t)$ (with $Z_k(t) \in \{0, 1\}$ and $P_k(t) \in \script{P}_k$ for all $k \in \{1, \ldots, K\}$). 
 The consumers then react by generating a 
 random demand vector $\bv{D}(t)$ with expectations given by (\ref{eq:demand-function}). 
 The actual number of products filled is scheduled by choosing the $\tilde{\bv{D}}(t)$ vector
 according to the scheduling 
 constraints (\ref{eq:scheduling-constraints1})-(\ref{eq:scheduling-constraints2}), and the resulting 
 queueing update is given by  (\ref{eq:dynamics}). 
 
For each $k \in \{1, \ldots, K\}$, define 
$\alpha_k$ as a fixed (non-negative) 
cost associated with assembling one product of type $k$. 
 Define a process $\phi(t)$ as follows: 
 \begin{equation} \label{eq:phi}
  \phi(t) \defequiv  - c(\bv{A}(t), X(t)) +  \sum_{k=1}^K Z_k(t)D_k(t)(P_k(t)-\alpha_k)
  \end{equation} 
  The value of $\phi(t)$ represents the total instantaneous profit due to material purchasing
  and product sales on slot $t$, under the assumption that all demands are fulfilled (so that 
  $\tilde{D}_k(t) = D_k(t)$ for all $k$).  Define $\phi_{actual}(t)$ as the \emph{actual} instantaneous
  profit, defined by replacing the $D_k(t)$ values in the right hand side of (\ref{eq:phi}) with 
  $\tilde{D}_k(t)$ values.  Note that $\phi(t)$ can be either positive, negative,
  or zero, as can $\phi_{actual}(t)$.

  Define time average expectations $\overline{\phi}$ and $\overline{\phi}_{actual}$ 
  as follows: 
  \[ \overline{\phi} \defequiv \lim_{t\rightarrow\infty} \frac{1}{t}\sum_{\tau=0}^{t-1} \expect{\phi(\tau)}  \: \: , 
  \: \: \overline{\phi}_{actual} \defequiv \lim_{t\rightarrow\infty} \frac{1}{t}\sum_{\tau=0}^{t-1} \expect{\phi_{actual}(\tau)} \]
  Every slot $t$, the plant observes the current queue vector $\bv{Q}(t)$, 
 the current demand state $Y(t)$, and the current supply state $X(t)$, 
 and chooses a purchase vector $\bv{A}(t) \in \script{A}(X(t))$ and 
 pricing vectors  $\bv{Z}(t)$, $\bv{P}(t)$ (with $Z_k(t) \in \{0, 1\}$ and $P_k(t) \in \script{P}_k$ for all $k \in \{1, \ldots, K\}$). 
 The consumers then react by generating a 
 random demand vector $\bv{D}(t)$ with expectations given by (\ref{eq:demand-function}). 
 The actual number of products filled is scheduled by choosing the $\tilde{\bv{D}}(t)$ vector
 according to the scheduling 
 constraints (\ref{eq:scheduling-constraints1})-(\ref{eq:scheduling-constraints2}), and the resulting 
 queueing update is given by  (\ref{eq:dynamics}). 
The goal of the plant is to maximize the time average expected profit $\overline{\phi}_{actual}$.  
  For convenience, a table of notation is given in Table \ref{table:notation}.

\begin{table}[ht]
\caption{Table of Notation} \centering
\begin{tabular}{l l} 
\hline Notation &  Definition \\ \hline
$X(t)$ & Supply state, $\pi(x) = Pr[X(t) = x]$ for $x \in \script{X}$\\
$\bv{A}(t) = (A_1(t), \ldots, A_M(t))$ & Raw material purchase vector for slot $t$ \\
$c(\bv{A}(t), X(t))$ & Raw material cost function \\
$\script{A}(X(t))$ & Constraint set for decision variables $\bv{A}(t)$   \\
$Y(t)$ & Consumer demand state, $\pi(y)=Pr[Y(t)=y]$ for $y \in \script{Y}$\\
$\bv{Z}(t) = (Z_1(t), \ldots, Z_K(t))$ & 0/1 sale vector  \\
$\bv{P}(t)=(P_1(t), \ldots, P_K(t))$ & Price vector, $P_k(t) \in \script{P}_k$\\
$\bv{D}(t)= (D_1(t), \ldots, D_K(t))$ & Random demand vector (in reaction to $\bv{P}(t)$) \\
$F_k(p,y)$ & Demand function,  $F_k(p, y) = \expect{D_k(t)\left|\right.P_k(t) = p, Y(t) = y}$\\
$\bv{Q}(t) = (Q_1(t), \ldots, Q_M(t))$ & Queue vector of raw materials in inventory \\
$Q_{m,max}$ & Maximum buffer size of queue $m$ \\
$\alpha_k$ & Cost incurred by assembly of one product of  type $k$  \\
$\beta_{mk}$ & Number of $m$ raw materials needed for assembly product type $k$ \\
$\phi(t)$ & Instantaneous profit variable for slot $t$ (given by (\ref{eq:phi})) \\
$\bv{\mu}(t) = (\mu_1(t), \ldots, \mu_M(t))$ & Departure vector for raw materials, 
$\mu_m(t) = \sum_{k=1}^K \beta_{mk}(t)D_k(t)$ \\
$\tilde{\bv{D}}(t), \tilde{\bv{\mu}}(t)$ & Actual fulfilled demands and raw materials used for slot $t$ \\
$\phi_{actual}(t)$ & Actual instantaneous profit for slot $t$ \\
[1ex] \hline
\end{tabular}
\label{table:notation}
\end{table}

\section{Characterizing Maximum Time Average Profit} \label{section:optimal-profit}

Assume infinite buffer capacity (so that $Q_{m,max} = \infty$ for all $m \in \{1, \ldots, M\}$). 
Consider any control algorithm that makes decisions for $\bv{Z}(t)$, $\bv{P}(t)$, $\bv{A}(t)$, 
and also makes scheduling decisions for $\tilde{\bv{D}}(t)$, 
according to the system structure as described in
the previous section.  Define $\phi^{opt}$ as the \emph{maximum time average profit} 
over all such algorithms, so that all algorithms must satisfy $\overline{\phi}_{actual} \leq \phi^{opt}$, but
there exist algorithms that can yield profit arbitrarily close to $\phi^{opt}$. 
The value of  $\phi^{opt}$ is determined by the steady state
distributions $\pi(x)$ and $\pi(y)$, the cost function $c(\bv{A}(t), X(t))$, and the demand
functions $F_k(p,y)$ according to the following theorem. 

\begin{thm} \label{thm:max-profit} (Maximum Time Average Profit)  Suppose the initial 
queue states satisfy $\expect{Q_m(0)} < \infty$ for all $m \in \{1, \ldots, M\}$. Then
under any control algorithm, the time average achieved profit satisfies: 
\[ \limsup_{t\rightarrow\infty} \frac{1}{t}\sum_{\tau=0}^{t-1} \expect{\phi_{actual}(\tau)} \leq \phi^{opt} \]
where $\phi^{opt}$ is the maximum value of the objective function in the following optimization 
problem, defined in terms of 
auxiliary variables  $\hat{c}$, $\hat{r}$, $\theta(\bv{a}, x)$,  $\hat{a}_m, \hat{\mu}_m$ (for all $x \in \script{X}, \bv{a} \in \script{A}(x)$, 
   $m \in \{1, \ldots, M\}$): 
\begin{eqnarray*}
\mbox{Maximize:} &&   \phi  \\
\mbox{Subject to:}  && \phi = -\hat{c} +  \hat{r}  \: \: \: , \: \: \:  \hat{a}_m \geq \hat{\mu}_m \: \: \forall m    \\
 \hat{c} &=& \sum_{x\in\script{X}} \pi(x)\sum_{\bv{a}\in\script{A}(x)} 
\theta(\bv{a}, x)c(\bv{a}, x)  \\
 \hat{r} &=& \sum_{y \in \script{Y}}\pi(y)\sum_{k=1}^K \expect{Z_k(t)(P_k(t)-\alpha_k)F_k(P_k(t), y)\left|\right. Y(t) = y}   \\
  \hat{a}_m &=& \sum_{x \in \script{X}}\pi(x)\sum_{\bv{a}\in\script{A}(x)}\theta(\bv{a}, x) a_m  \: \: \forall m   \\
 \hat{\mu}_m  &=& \sum_{y \in \script{Y}}\pi(y)\sum_{k=1}^K \beta_{mk} \expect{Z_k(t)F_k(P_k(t), y)\left|\right.Y(t) = y}  \: \: \forall m  \\
 && \hspace{-.3in} 0 \leq \theta(\bv{a}, x)  \leq 1 \: \: \:  \forall x \in \script{X} ,  \bv{a} \in \script{A}(x)  \\
 && \hspace{-.3in} \sum_{\bv{a}\in\script{A}(x)}\theta(\bv{a}, x) = 1 \: \: \: \forall x \in \script{X}  \\
 && \hspace{-.3in} P_k(t) \in \script{P} \: \: , \: \: Z_k(t) \in \{0, 1\} \: \: \: \forall k , t 
\end{eqnarray*}
where $\bv{P}(t) = (P_1(t), \ldots, P_K(t))$ and $\bv{Z}(t) = (Z_1(t), \ldots, Z_K(t))$ are vectors 
randomly chosen with a conditional distribution that can be chosen as any 
distribution that depends only on the observed value of 
$Y(t) = y$. The expectations in the above problem are with respect to the chosen conditional distributions
for these decisions. 
\end{thm}

\begin{proof} 
See Appendix A.
\end{proof}

In Section \ref{section:dynamic} 
we show that algorithms can be designed to achieve a time average profit $\overline{\phi}$ that is 
within $\epsilon$ of the value $\phi^{opt}$ defined in Theorem \ref{thm:max-profit}, for any arbitrarily small 
$\epsilon>0$.  Thus, $\phi^{opt}$ represents the optimal time average profit over all possible algorithms. 

The variables in Theorem \ref{thm:max-profit} can be interpreted as follows:   The variable 
$\theta(\bv{a}, x)$ 
represents a conditional probability of choosing $\bv{A}(t) = \bv{a}$ given that the plant
observes supply state $X(t) = x$.  The variable $\hat{c}$ thus represents the time average cost of purchasing raw materials
under this stationary randomized policy, and the variable $\hat{r}$ represents the time
average revenue for selling products. The variables $\hat{a}_m$ and $\hat{\mu}_m$ represent
the time average arrival and departure rates for queue $m$, respectively. 
The above theorem thus characterizes $\phi^{opt}$ in terms of 
all possible \emph{stationary randomized control algorithms}, that is, all algorithms
that  make  randomized choices for $\bv{A}(t), \bv{Z}(t), \bv{P}(t)$ according to fixed
conditional distributions given the supply state $X(t)$ and demand state $Y(t)$.
Note that Theorem \ref{thm:max-profit} contains 
no variables for the  scheduling decisions for $\tilde{\bv{D}}(t)$, made
subject to (\ref{eq:scheduling-constraints1})-(\ref{eq:scheduling-constraints2}).  
Such scheduling decisions allow choosing $\tilde{\bv{D}}(t)$ in reaction
to the demands $\bv{D}(t)$,  and hence allow more flexibility beyond the choice of the 
$\bv{Z}(t)$ and $\bv{P}(t)$ variables alone (which must be chosen before the demands
$\bv{D}(t)$ are observed).  That such  additional scheduling  options
cannot be exploited to increase time average profit  is a consequence of our
proof of Theorem \ref{thm:max-profit}.


We say that a policy 
is \emph{$(X,Y)$-only} if it chooses $\bv{P}(t)$, $\bv{Z}(t)$, $\bv{A}(t)$ values as a stationary and randomized
function only of the current observed $X(t)$ and $Y(t)$ states. 
Because the sets $\script{P}_k$ are compact and the functions $F_k(p,y)$ are continuous
in $p \in \script{P}_k$ for each $y \in \script{Y}$,  it can be shown that the value of $\phi^{opt}$ in 
Theorem \ref{thm:max-profit} can be \emph{achieved} by a particular $(X,Y)$-only policy, as shown
in the following corollary.

\begin{cor} \label{cor:1} There exists an $(X,Y)$-only policy $\bv{P}^*(t)$, $\bv{Z}^*(t)$, $\bv{A}^*(t)$
such that: \footnote{Note that in (\ref{eq:a-stat1} ) we have changed the ``$\geq$'' into ``$=$''. It is easy to show that doing so in Theorem \ref{thm:max-profit} does not result in any loss of optimality.}
\begin{eqnarray} 
 \expect{\phi^*(t)} &=& \phi^{opt} \label{eq:phi-stat1}  \\
\expect{A_m^*(t)} &=& \expect{\mu_m^*(t)} \forall m \in \{1, \ldots, M\} \label{eq:a-stat1} 
\end{eqnarray}
 where $\phi^{opt}$ is the optimal time average profit defined in Theorem \ref{thm:max-profit}, and
where $\expect{\phi^*(t)}$ and $\expect{\mu_m^*(t)}$ are  given by: 
\begin{eqnarray*}
 \expect{\phi^*(t)} &=& -\expect{c(\bv{A}^*(t), X(t))} + \sum_{k=1}^K\expect{Z_k^*(t)(P_k^*(t)-\alpha_k)F_k(P_k^*(t), Y(t))} \\
 \expect{\mu_m^*(t)} &=& \sum_{k=1}^K \beta_{mk} \expect{Z_k^*(t)F_k(P_k^*(t), Y(t))} \: \: \forall m \in \{1,\ldots, M\}
 \end{eqnarray*}
where 
the expectations are with respect to the stationary probability distributions 
$\pi(x)$ and $\pi(y)$ for $X(t)$ and $Y(t)$, and the (potentially randomized) decisions 
for $\bv{A}^*(t), \bv{Z}^*(t), \bv{P}^*(t)$ that depend on $X(t)$ and $Y(t)$.  
\end{cor}

\subsection{On the Sufficiency of Two Prices}

It can be shown that the $(X,Y)$-only policy of Corollary \ref{cor:1} can be used to achieve time average
profit arbitrarily close to optimal as follows:  
Define a parameter $\rho$ such that $0 < \rho < 1$.  Use the $(X,Y)$-only decisions for $\bv{P}^*(t)$
and $\bv{A}^*(t)$ every slot $t$, but use new decisions $\tilde{Z}_k(t) = Z_k^*(t)1_k(t)$, where $1_k(t)$ 
is an i.i.d. Bernoulli process with $Pr[1_k(t)  =1] = \rho$.  It follows that the inequality (\ref{eq:a-stat1}) 
becomes: 
\[ \expect{A_m^*(t)} =  \expect{\mu_m^*(t)} = (1/\rho)\expect{\tilde{\mu}_m(t)}  \]
where $\tilde{\mu}_m(t)$ corresponds to the new decisions $\tilde{Z}_k(t)$. 
It follows that all queues with non-zero arrival rates $\expect{A_m^*(t)}$ have these rates \emph{strictly greater}
than the expected service rates $\expect{\tilde{\mu}_m(t)}$, and so these queues grow to infinity with probability 
1. It follows that we always have enough material to meet the consumer demands, so that $\tilde{\bv{D}}(t) = \bv{D}(t)$
and the scheduling decisions (\ref{eq:scheduling-constraints1})-(\ref{eq:scheduling-constraints2}) become
irrelevant.  This reduces profit only by a factor $O(1-\rho)$, which can be made arbitrarily small as $\rho \rightarrow 1$. 

Here we show that the $(X,Y)$-only policy of Corollary \ref{cor:1}
can be changed into an $(X,Y)$-only  policy that 
randomly chooses between at most \emph{two} prices for each unique 
product $k \in \{1, \ldots, K\}$ and each unique demand state $Y(t) \in \script{Y}$,
while still satisfying (\ref{eq:phi-stat1})-(\ref{eq:a-stat1}).  This result is based on a similar two-price
theorem  derived  in \cite{two-prices-allerton07} for the case of a service provider with a single
queue.  We extend the result here to the case of a product provider with multiple queues. 

\begin{thm} \label{thm:two-price} Suppose there exists an $(X,Y)$-only algorithm  that allocates
$\bv{Z}(t)$ and $\bv{P}(t)$ to yield (for some given values $\hat{r}$ and $\hat{\mu}_m$ for 
$m \in \{1,\ldots, M\}$): 
\begin{eqnarray}
&\sum_{k=1}^K \expect{Z_k(t)(P_k(t)-\alpha_k)F_k(P_k(t), Y(t))} \geq \hat{r}& \label{eq:two-price1} \\
&\sum_{k=1}^K \beta_{mk} \expect{Z_k(t) F_k(P_k(t), Y(t))} \leq \hat{\mu}_m \: \: \mbox{ for all $m \in \{1, \ldots, M\}$}& \label{eq:two-price2} 
\end{eqnarray}
Then the same inequality constraints can be achieved by a new stationary randomized 
policy $\bv{Z}^*(t)$, $\bv{P}^*(t)$
that uses at most two prices for each unique product $k \in \{1, \ldots, K\}$ and each unique
demand state $Y(t) \in \script{Y}$. 
\end{thm}

 \begin{proof} The proof is given in Appendix B.
 \end{proof}

The expectation on the left hand side of (\ref{eq:two-price1}) represents the expected revenue generated
from sales under the original $(X,Y)$-only policy, and the expectation on the left hand side of (\ref{eq:two-price2}) 
represents the expected departures from queue $Q_m(t)$ under this policy.  The theorem says that the 
pricing part of the $(X,Y)$-only
algorithm, which potentially uses many different price options,  
can be changed to a 2-price algorithm without decreasing revenue
or increasing demand for materials.

 
 
 Simple examples can be given to show that two prices are often \emph{necessary}
 to achieve maximum time average profit, even when user demand functions are the same
 for all slots (see \cite{two-prices-allerton07} for a simple  example for the related
 service-provider problem).   We emphasize that the $(X,Y)$-only policy of Corollary \ref{cor:1} 
is not necessarily  practical, as implementation would require 
full knowledge of the $\pi(x)$ and $\pi(y)$ distributions, and it would require 
a solution to the (very complex) optimization problem of Theorem \ref{thm:max-profit} even if
the $\pi(x)$ and $\pi(y)$ distributions were known.
Further, it relies on having an infinite buffer capacity 
(so that $Q_{m,max} = \infty$ for all $m \in\{1, \ldots, M\}$).
A more practical algorithm is developed in the next section 
that overcomes these difficulties.

 \section{A  Dynamic Pricing and Purchasing Algorithm}  \label{section:dynamic}

 Here we construct a dynamic algorithm that makes purchasing and pricing 
decisions in reaction to the current queue sizes and the observed $X(t)$ and 
$Y(t)$ states, without knowledge of the $\pi(x)$ and $\pi(y)$ probabilities  that govern
the evolution of these
states.  We begin with the assumption that $X(t)$ is i.i.d. over slots with probabilities $\pi(x) = Pr[X(t) = x]$, 
and $Y(t)$ is i.i.d. over slots with $\pi(y) = Pr[Y(t)=y]$. This assumption is extended to more general
non-i.i.d. processes in Sections \ref{section:ergodic} and \ref{section:non-ergodic}. 

 Define $1_k(t)$ as an indicator variable that is $1$ if and only 
 if $Q_m(t)< \mu_{m, max}$ for some queue $m$ such that $\beta_{mk}>0$ (so that type $m$
 raw material is used to create product $k$): 
 \begin{equation} \label{eq:edge}  
 1_k(t)  = \left\{ \begin{array}{ll}
                          1 &\mbox{ if $Q_m(t) < \mu_{m, max}$ for some  $m$ such that $\beta_{mk}>0$} \\
                             0  & \mbox{ otherwise} 
                            \end{array}
                                 \right.
                                 \end{equation} 
                                 
To begin, let us choose an algorithm from the restricted class of algorithms that choose $\bv{Z}(t)$ values 
to satisfy the following
\emph{edge constraint} every slot $t$: 
\begin{eqnarray} 
\mbox{For all $k \in\{1, \ldots, K\}$, we have: } Z_k(t) = 0 \mbox{ whenever } 1_k(t) = 1 \label{eq:edge-constraint} 
\end{eqnarray} 
 That is, the edge constraint (\ref{eq:edge-constraint}) ensures that 
 no type $k$ product is sold unless
 inventory in all of its corresponding raw material queues $m$ is at least $\mu_{m, max}$. 
   Under this restriction, we always have enough raw material for any generated demand vector, 
   and so 
 $\tilde{D}_k(t) = D_k(t)$  for all products $k$ and all slots $t$.  Thus, 
 from (\ref{eq:phi}) we have  $\phi_{actual}(t) = \phi(t)$.  
 Define 
 $\mu_m(t)$ as the number of material queue $m$ departures on slot $t$: 
 \begin{equation} \label{eq:mu} 
 \mu_m(t) \defequiv \sum_{k=1}^K \beta_{mk}Z_k(t)D_k(t)
 \end{equation}
 The queueing dynamics of (\ref{eq:dynamics})
 thus become: 
 \begin{equation} \label{eq:dynamics1} 
  Q_m(t+1) = Q_m(t) - \mu_m(t) + A_m(t) 
  \end{equation} 
The above equation continues to assume we have infinite buffer space, but we soon show that we
need only a finite buffer to implement our solution.   
 
 \subsection{Lyapunov Drift} 
 
 For a given set of non-negative 
 parameters $\{\theta_m\}$ for  $m \in \{1, \ldots, M\}$, 
define the non-negative \emph{Lyapunov function} $L(\bv{Q}(t))$ as follows: 
\begin{eqnarray}
L(\bv{Q}(t)) \defequiv  \frac{1}{2}\sum_{m=1}^M (Q_m(t)-\theta_m)^2 \label{eq:lyap-function} 
\end{eqnarray}
This Lyapunov function is similar to that used for stock trading problems in \cite{neely-stock-arxiv}, 
and has the flavor of keeping queue backlog near a non-zero value $\theta_m$, as in \cite{neely-energy-delay-it}. 
Define the conditional Lyapunov drift $\Delta(\bv{Q}(t))$ as follows:\footnote{Strictly speaking, 
we should use the notation $\Delta(\bv{Q}(t), t)$, as the drift may be non-stationary. 
However, we use the simpler notation $\Delta(\bv{Q}(t))$ as a formal representation of the 
right hand side of (\ref{eq:lyap}).}
\begin{equation} \label{eq:lyap} 
\Delta(\bv{Q}(t)) \defequiv \expect{L(\bv{Q}(t+1)) - L(\bv{Q}(t)) \left|\right.\bv{Q}(t)} 
\end{equation} 
Define a constant $V>0$, to be used to  affect the revenue-storage tradeoff.  Using
the stochastic optimization technique of \cite{now}, our approach is to design a strategy
that, every slot $t$,  observes current system conditions $\bv{Q}(t)$, $X(t)$, $Y(t)$ 
and makes pricing and purchasing decisions to 
minimize a bound on the following ``drift-plus-penalty'' expression: 
\[ \Delta(\bv{Q}(t)) - V\expect{\phi(t)\left|\right.\bv{Q}(t)}  \]
where $\phi(t)$ is the instantaneous profit function defined in (\ref{eq:phi}). 
%

\subsection{Computing the Drift}

We have the following lemma. 
\begin{lem} \label{lem:drift-comp}  (Drift Computation) Under any algorithm that satisfies the 
edge constraint (\ref{eq:edge-constraint}), and for any constants $V\geq0$, $\theta_m \geq 0$ 
for $m \in\{1, \ldots, M\}$, 
the Lyapunov drift $\Delta(\bv{Q}(t))$ satisfies: 
\begin{eqnarray}
  \Delta(\bv{Q}(t)) - V\expect{\phi(t)|\bv{Q}(t)}  &\leq& B - V\expect{\phi(t)|\bv{Q}(t)} \nonumber \\
  && + \sum_{m=1}^M(Q_m(t) - \theta_m)\expect{A_m(t) - \mu_m(t) \left|\right.\bv{Q}(t)} \label{eq:drift-q}
\end{eqnarray}
where the constant $B$ is defined: 
\begin{equation} \label{eq:B} 
B \defequiv \frac{1}{2}\sum_{m=1}^M\max[A_{m,max}^2, \mu_{m,max}^2] 
\end{equation}  
\end{lem}
\begin{proof}  The edge constraint (\ref{eq:edge-constraint}) ensures that the dynamics (\ref{eq:dynamics1}) 
hold for all $t$. 
By squaring  (\ref{eq:dynamics1}) we have: 
\[ (Q_m(t+1) - \theta_m)^2  =  (Q_m(t) - \theta_m)^2 + (A_m(t)-\mu_m(t))^2  + 2(Q_m(t)-\theta_m)(A_m(t) - \mu_m(t)) \]
Dividing by $2$, summing over $m \in \{1, \ldots, M\}$, and taking conditional expectations yields: 
\[ \Delta(\bv{Q}(t)) = \expect{B(t)|\bv{Q}(t)} + \sum_{m=1}^M(Q_m(t) - \theta_m)\expect{A_m(t) - \mu_m(t)|\bv{Q}(t)} \]
where $B(t)$ is defined: 
\begin{equation} \label{eq:Bt-appendix} 
B(t) \defequiv \frac{1}{2}\sum_{m=1}^M(A_m(t) - \mu_m(t))^2 
\end{equation} 
By the finite bounds on $A_m(t)$ and $\mu_m(t)$, we clearly have $B(t) \leq B$ for all $t$. 
\end{proof} 

  Now note from the definition of $\mu_m(t)$ in (\ref{eq:mu}) that: 
\begin{eqnarray}
\expect{\mu_m(t) \left|\right.\bv{Q}(t)} &=& \expect{ \sum_{k=1}^K \beta_{mk}Z_k(t)D_k(t) \left|\right.\bv{Q}(t)} \nonumber \\
&=& \sum_{k=1}^K \beta_{mk}\expect{\expect{Z_k(t)D_k(t)\left|\right.\bv{Q}(t), P_k(t), Y(t)}\left|\right.\bv{Q}(t)} \nonumber \\
&=& \sum_{k=1}^K\beta_{mk}\expect{Z_k(t)F_k(P_k(t), Y(t))\left|\right.\bv{Q}(t)} \label{eq:foo} 
\end{eqnarray}
where we have used the law of iterated expectations in the final equality. Similarly, we have from the
definition of $\phi(t)$ in (\ref{eq:phi}): 
\begin{eqnarray}
\expect{\phi(t) \left|\right.\bv{Q}(t)} &=& -\expect{c(\bv{A}(t), X(t))\left|\right.\bv{Q}(t)} \nonumber \\
&& + \sum_{k=1}^K \expect{Z_k(t)(P_k(t)-\alpha_k)F_k(P_k(t), Y(t))\left|\right.\bv{Q}(t)} \label{eq:foo2} 
\end{eqnarray} 
 Plugging (\ref{eq:foo}) and (\ref{eq:foo2}) into (\ref{eq:drift-q}) yields: 
  \begin{eqnarray}
 \Delta(\bv{Q}(t)) - V\expect{\phi(t)\left|\right.\bv{Q}(t)} &\leq& B \nonumber \\
&& \hspace{-1.8in} +\sum_{m=1}^M (Q_m(t) - \theta_m)\expect{A_m(t) - \sum_{k=1}^K\beta_{mk}Z_k(t)F_k(P_k(t), Y(t)) \left|\right.\bv{Q}(t)} \nonumber \\
&& \hspace{-1.8in} +  V\expect{c(\bv{A}(t), X(t))\left|\right.\bv{Q}(t)} \nonumber \\
&& \hspace{-1.8in} - V \expect{  \sum_{k=1}^K Z_k(t)(P_k(t)-\alpha_k) F_k(P_k(t), Y(t))   \left|\right.\bv{Q}(t)} \label{eq:big-drift2} 
 \end{eqnarray}
 In particular, the right hand side of (\ref{eq:big-drift2}) is identical to the right hand side of (\ref{eq:drift-q}).

\subsection{The Dynamic Purchasing and Pricing Algorithm}

 Minimizing the right hand side of the (\ref{eq:big-drift2}) 
 over all feasible choices of $\bv{A}(t)$, 
 $\bv{Z}(t)$, $\bv{P}(t)$ (given the observed $X(t)$ and $Y(t)$ states and the observed
 queue values $\bv{Q}(t)$) yields the following algorithm:
 
 \emph{\underline{Joint Purchasing and Pricing Algorithm (JPP)}:} 
Every slot $t$,  perform the 
following actions: 
\begin{enumerate}
\item \emph{Purchasing:} 
Observe $\bv{Q}(t)$ and $X(t)$, and choose $\bv{A}(t)=(A_1(t), \ldots, A_M(t))$ as the solution  of the following
optimization problem defined for slot $t$:
\begin{eqnarray}
\mbox{Minimize:} &  Vc(\bv{A}(t), X(t))  + \sum_{m=1}^M A_m(t)(Q_m(t)-\theta_m)  \label{eq:purchase1}  \\
\mbox{Subject to:} & \bv{A}(t) \in \script{A}(X(t)) \label{eq:purchase2} 
\end{eqnarray}
where $\script{A}(X(t))$ is defined by constraints (\ref{eq:Amax})-(\ref{eq:cmax}).

\item \emph{Pricing:} Observe $\bv{Q}(t)$ and $Y(t)$. For each product $k \in \{1, \ldots, K\}$,  if $1_k(t) = 1$, choose $Z_k(t) = 0$
and do not offer product $k$ for sale.  If $1_k(t) = 0$, choose $P_k(t)$
as the solution to the following problem: 
\begin{eqnarray}
\mbox{Maximize:} && V(P_k(t)-\alpha_k)F_k(P_k(t), Y(t))  \nonumber \\
&& + F_k(P_k(t), Y(t))\sum_{m=1}^M\beta_{mk}(Q_m(t)-\theta_m) \label{eq:pricing1}  \\
\mbox{Subject to:} && P_k(t) \in \script{P}_k \label{eq:pricing2} 
\end{eqnarray}
If the above maximization is  positive, set $Z_k(t) = 1$ and keep $P_k(t)$
as the above value.  Else, set $Z_k(t) = 0$ and do not offer product $k$ for sale. 

\item \emph{Queue Update:} Fulfill all demands $D_k(t)$, and update the
queues $Q_m(t)$ according to (\ref{eq:dynamics1}) (noting by construction of this 
algorithm that $\tilde{\bv{D}}(t) = \bv{D}(t)$ for all $t$, so that the dynamics (\ref{eq:dynamics1}) 
are equivalent to (\ref{eq:dynamics})). 
\end{enumerate}

The above JPP algorithm does not require knowledge of probability distributions
$\pi(x)$ or $\pi(y)$, and is decoupled into separate policies for pricing and purchasing. 
The  pricing policy is quite simple and involves maximizing a (possibly
non-concave) function of one variable $P_k(t)$ over the 1-dimensional 
set $\script{P}_k$.  For example, if  $\script{P}_k$ is a discrete set of 1000 price options, this involves
evaluating the function over each option and choosing the maximizing price. The 
purchasing policy is more complex, as $\bv{A}(t)$ is an integer vector that must satisfy 
(\ref{eq:Amax})-(\ref{eq:cmax}). This is a \emph{knapsack problem} due to the constraint (\ref{eq:cmax}). 
However, the decision is trivial in the case when $c_{max} = \sum_{m=1}^Mx_{m,max}A_{m,max}$,
in which case the constraint (\ref{eq:cmax}) is redundant.

\subsection{Deterministic Queue Bounds} 

We have the following simple lemma that shows the above policy can be implemented
on a finite buffer system. 
\begin{lem} \label{lem:finite-buffer} (Finite Buffer Implementation) 
If initial inventory satisfies $Q_m(0) \leq Q_{m, max}$ for all $m \in \{1, \ldots, M\}$, where
$Q_{m, max} \defequiv \theta_m + A_{m,max}$, 
then JPP yields  $Q_m(t) \leq Q_{m, max}$ for all slots $t \geq 0$ and all queues $m\in\{1, \ldots, M\}$. 
\end{lem}
\begin{proof} 
Fix a queue $m$, and suppose that $Q_m(t) \leq Q_{m, max}$ for some slot $t$ (this holds by assumption at $t=0$). 
We show that $Q_m(t+1) \leq Q_{m, max}$.  To see this, note that because the 
function  $c(\bv{A}(t), X(t))$ in (\ref{eq:example-cost}) 
is non-decreasing in every entry of $\bv{A}(t)$ for each 
$X(t)\in \script{X}$, the purchasing policy (\ref{eq:purchase1})-(\ref{eq:purchase2})
yields $A_m(t) = 0$ for any queue $m$ that satisfies
$Q_m(t) > \theta_m$.    
It follows that $Q_m(t)$ cannot increase if it is greater than $\theta_m$, 
and so $Q_m(t+1) \leq \theta_m  + A_{m, max}$ (because $A_{m,max}$ is the maximum amount of 
increase for queue $m$ on any slot).
 \end{proof} 
 
 The following important related lemma shows that queue sizes $Q_m(t)$ are always above $\mu_{m, max}$,
  provided that they start with at least this much raw
  material and that the $\theta_m$ values are chosen to be
  sufficiently large.  Specifically, define $\theta_m$ as follows: 
   \begin{equation} \label{eq:good-theta} 
  \theta_m \defequiv \max_{\{k \in\{1, \ldots, K\} | \beta_{mk}>0\}} \left[\frac{V(P_{k,max} - \alpha_k)}{\beta_{mk}} + 
  \sum_{i\in\{1, \ldots, M\} , i \neq m} \frac{\beta_{ik}A_{i,max}}{\beta_{mk}} + 2\mu_{m,max}   \right]
 \end{equation} 
  \begin{lem} \label{lem:lower-bound} Suppose that $\theta_m$ is defined by (\ref{eq:good-theta}), 
  and that  $Q_m(0) \geq \mu_{m,max}$ 
  for all $m \in \{1, \ldots, M\}$. Then for all slots $t\geq0$ we have: 
 \[ Q_m(t) \geq \mu_{m,max}  \: \: \forall m \in\{1, \ldots, M\} \]
  \end{lem}  
  \begin{proof} 
  Fix $m \in \{1, \ldots, M\}$, and 
    suppose that $Q_m(t) \geq \mu_{m,max}$ for some slot $t$ (this holds by assumption for $t=0$). We
  prove it also holds for slot $t+1$.  If $Q_m(t) \geq 2\mu_{m,max}$, then $Q_m(t+1) \geq \mu_{m,max}$
  (because at most $\mu_{m,max}$ units can depart queue $m$ on any slot), and hence we are done. 
  Suppose now that $\mu_{m,max} \leq Q_m(t) < 2\mu_{m,max}$. 
  In this case the pricing functional in  (\ref{eq:pricing1}) satisfies the following for any 
   product $k$ such that 
  $\beta_{mk} >0$:  
  \begin{eqnarray}
  &&\hspace{-.3in}V(P_k(t) - \alpha_k)F_k(P_k(t), Y(t)) +  F_k(P_k(t), Y(t))\sum_{i=1}^M\beta_{ik}(Q_i(t)-\theta_i)  \nonumber \\
  &\leq& F_k(P_k(t),Y(t)) \times \nonumber \\
  && \left[V(P_{k,max} - \alpha_k) + \sum_{i\in\{1, \ldots, M\} ,i \neq m}\beta_{ik}A_{i,max} + \beta_{mk}(2\mu_{m,max}-\theta_m)\right] \label{eq:edge2} \\
  &\leq& 0 \label{eq:edge3} 
  \end{eqnarray}
  where in (\ref{eq:edge2}) 
  we have used the fact that $(Q_i(t) - \theta_i) \leq A_{i,max}$ for all queues $i$ (and in particular 
  for all $i \neq m$) by 
  Lemma \ref{lem:finite-buffer}.  In (\ref{eq:edge3}) we have used the definition of $\theta_m$ in (\ref{eq:good-theta}). 
  It follows that the pricing rule (\ref{eq:pricing1})-(\ref{eq:pricing2}) sets $Z_k(t) = 0$ for all products $k$ that
  use raw material $m$, and so no departures can take place from $Q_m(t)$ on the current slot. 
  Thus: $\mu_{m,max} \leq Q_m(t) \leq Q_m(t+1)$, and we are done. 
  \end{proof}

\subsection{Performance Analysis of JPP}

 \begin{thm} \label{thm:performance1} (Performance of JPP)  
 Suppose that $\theta_m$ is defined by (\ref{eq:good-theta}) for all $m \in \{1, \ldots, M\}$, 
 and that $\mu_{m,max} \leq Q_m(0) \leq \theta_m + A_{m,max}$ for all $m \in \{1, \ldots, M\}$.  Suppose
 $X(t)$ and $Y(t)$ are i.i.d. over slots. 
 Then under the JPP algorithm implemented with any parameter $V>0$, we have: 
 
 (a)  For all $m \in \{1, \ldots, M\}$ and all slots $t \geq 0$: 
 \[ \mu_{m,max} \leq Q_m(t) \leq Q_{m, max} \defequiv \theta_{m,max} + A_{m,max}  \]
 where $Q_{m,max} = O(V)$. 
 
 (b) For all slots $t>0$ we have: 
\begin{equation} 
\frac{1}{t}\sum_{\tau=0}^{t-1}\expect{\phi_{actual}(\tau)} \geq \phi^{opt} - \frac{B}{V} - \frac{\expect{L(\bv{Q}(0))}}{Vt} 
\end{equation}  
 where the constant $B$ is defined in (\ref{eq:B}) and is independent of $V$, and where
 $\phi^{opt}$ is the optimal time average profit defined in Theorem \ref{thm:max-profit}.  
 
(c) The time average profit converges with probability 1, and satisfies: 
\begin{equation} 
\lim_{t\rightarrow\infty} \frac{1}{t}\sum_{\tau=0}^{t-1} \phi_{actual}(\tau) \geq \phi^{opt} - \frac{B}{V} \: \: \: \: \mbox{(with probability 1)} 
\end{equation}  
\end{thm} 
 
 Thus, the time average profit is within $O(1/V)$ of optimal, and hence can be pushed arbitrarily close
 to optimal by increasing $V$, with a tradeoff in the maximum
 buffer size that is $O(V)$.    Defining $\epsilon = B/V$ yields the desired $[O(\epsilon), O(1/\epsilon)]$ profit-buffer
 size tradeoff. 
   
\begin{proof}  (Theorem \ref{thm:performance1} parts (a) and (b)) 
Part (a) follows immediately from Lemmas \ref{lem:finite-buffer} and \ref{lem:lower-bound}. 
To prove part (b), note that JPP observes $\bv{Q}(t)$ and 
makes control decisions for $Z_k(t)$, $\bv{A}(t)$, $P_k(t)$ 
that minimize the right hand side of (\ref{eq:drift-q}) 
under any alternative choices. Thus: 
  \begin{eqnarray}
 \Delta(\bv{Q}(t)) - V\expect{\phi(t)\left|\right.\bv{Q}(t)} &\leq& B - V\expect{\phi^*(t)|\bv{Q}(t)}  \nonumber \\
 && + \sum_{m=1}^M(Q_m(t)-\theta_m)\expect{A_m^*(t)-\mu_m^*(t)|\bv{Q}(t)} 
 \label{eq:big-drift3} 
 \end{eqnarray}
 where $\expect{\phi^*(t)|\bv{Q}(t)}$, and $\expect{\mu_m^*(t)|\bv{Q}(t)}$ correspond to 
 any alternative choices for the decision variables 
 $Z_k^*(t)$, $P_k^*(t)$, $A_m^*(t)$ subject to the same constraints, being that 
 $P_k^*(t) \in \script{P}_k$, $A_m^*(t)$ satisfies (\ref{eq:Amax})-(\ref{eq:cmax}), and 
 $Z_k^*(t) \in \{0, 1\}$, and $Z_k^*(t) = 0$ whenever $1_k(t) = 1$. Because $Q_m(t) \geq \mu_{m,max}$ for all
 $m$, we have $1_k(t)=0$ for all $k \in \{1, \ldots, K\}$ (where $1_k(t)$ is defined in (\ref{eq:edge})).  Thus, 
 the $(X,Y)$-only policy of Corollary \ref{cor:1} satisfies the desired constraints.  Further, this policy 
 makes decisions based only on $(X(t), Y(t))$, which are i.i.d. over slots and hence independent
 of the current queue state $\bv{Q}(t)$.  Thus, from (\ref{eq:phi-stat1})-(\ref{eq:a-stat1}) we have: 
 \begin{eqnarray*}
 \expect{\phi^*(t)|\bv{Q}(t)}  &=& \expect{\phi^*(t)} = \phi^{opt} \\
 \expect{A_m^*(t) - \mu_m^*(t)|\bv{Q}(t)} &=& \expect{A_m^*(t) - \mu_m^*(t)} = 0 \: \: \forall m \in \{1, \ldots, M\} 
 \end{eqnarray*}
 Plugging the above two identities into the right hand side of  (\ref{eq:big-drift3}) yields:  
   \begin{eqnarray}
 \Delta(\bv{Q}(t)) - V\expect{\phi(t)\left|\right.\bv{Q}(t)} &\leq& B -V\phi^{opt}
\label{eq:big-drift4} 
 \end{eqnarray}
 Taking expectations of the above and using the definition of $\Delta(\bv{Q}(t))$ yields: 
 \[ \expect{L(\bv{Q}(t+1))} - \expect{L(\bv{Q}(t))} - V\expect{\phi(t)} \leq B - V\phi^{opt} \]
The above holds for all slots $t\geq0$.  Summing over $\tau\in \{0, 1, \ldots, t-1\}$ for some integer $t>0$
yields: 
\begin{equation*} 
 \expect{L(\bv{Q}(t))} - \expect{L(\bv{Q}(0))} - V\sum_{\tau=0}^{t-1}\expect{\phi(\tau)} \leq Bt- Vt\phi^{opt} 
 \end{equation*} 
Dividing by $tV$, rearranging terms, and using non-negativity of $L(\bv{Q}(t))$ yields: 
\begin{equation} \label{eq:here} 
 \frac{1}{t}\sum_{\tau=0}^{t-1} \expect{\phi(\tau)} \geq \phi^{opt} - \frac{B}{V} - \frac{\expect{L(\bv{Q}(0))}}{Vt} 
 \end{equation} 
 Because $1_k(t) = 0$ for all $k$ and all $\tau$, we have $\phi(\tau) = \phi_{actual}(\tau)$ for all $\tau$, and we
are done. 
\end{proof}

\begin{proof} (Theorem \ref{thm:performance1} part (c)) Taking a limit of (\ref{eq:here}) proves that: 
\[ \liminf_{t\rightarrow\infty} \frac{1}{t}\sum_{\tau=0}^{t-1} \expect{\phi(\tau)} \geq \phi^{opt} - B/V \]
The above result made no assumption on the initial distribution of $\bv{Q}(0)$. Thus, letting $\bv{Q}(0)$ be any particular initial state, we have: 
\begin{equation} \label{eq:recurrent} 
 \liminf_{t\rightarrow\infty} \frac{1}{t}\sum_{\tau=0}^{t-1}\expect{\phi(\tau)|\bv{Q}(0)} \geq \phi^{opt} - B/V 
 \end{equation} 
However, under the JPP algorithm
the process  
$\bv{Q}(t)$ is a discrete time Markov chain with finite state space (because it is an integer valued vector
with finite bounds given in part (a)).  
Suppose now the initial condition $\bv{Q}(0)$ is a recurrent state. It follows that the time average of
$\phi(t)$ must converge to a 
well defined constant $\overline{\phi}(\bv{Q}(0))$  with probability 1 (where the constant may depend on the initial recurrent
state $\bv{Q}(0)$ that is chosen).  
Further, because $\phi(\tau)$ is bounded above and below by finite constants for all 
$\tau$, by the Lebesgue dominated convergence theorem  we have: 
\[ \lim_{t\rightarrow\infty} \frac{1}{t}\sum_{\tau=0}^{t-1}\expect{\phi(\tau)|\bv{Q}(0)} = \overline{\phi}(\bv{Q}(0)) \]
Using this in (\ref{eq:recurrent}) yields: 
\[ \overline{\phi}(\bv{Q}(0)) \geq \phi^{opt} - B/V \]
This is true for all initial states that are recurrent.  If $\bv{Q}(0)$ is a transient state, then $\bv{Q}(t)$ eventually
reaches a recurrent state and hence achieves a time average that is, with probability 1, greater than or 
equal to $\phi^{opt} - B/V$.  
\end{proof} 

\subsection{Place-Holder Values} 

Theorem \ref{thm:performance1} seems to require the initial queue values to satisfy
$Q_m(t) \geq \mu_{m,max}$ for all $t$.  This suggests that we need to purchase that many
raw materials before start-up. Here we use the \emph{place holder backlog} 
technique of \cite{neely-asilomar08} to show that we can achieve the same performance without this
initial start-up cost, without loss of optimality.   The technique also allows us to reduce our maximum
buffer size requirement $Q_{m,max}$ by an amount $\mu_{m,max}$ for all $m \in \{1, \ldots, M\}$, with
no loss in performance. 

To do this, we start the system off with exactly $\mu_{m,max}$
units of \emph{fake raw material} in each queue $m \in \{1, \ldots, M\}$.  Let $Q_m(t)$ be the total raw
material in queue $m$, including both the actual and fake material.  Let $Q^{actual}_m(t)$ be the amount
of actual raw material in queue $m$.  Then for slot $0$ we have: 
\[ Q_m(0) = Q_m^{actual}(0) + \mu_{m,max} \: \: \: \forall m \in \{1, \ldots, M\} \]
Assume that $\mu_{m,max} \leq Q_m(0) \leq \theta_m + A_{m,max}$ for all $m \in \{1, \ldots, M\}$, as needed
for Theorem \ref{thm:performance1}. Thus, $0 \leq Q_m^{actual}(0) \leq \theta_m + A_{m,max} - \mu_{m,max}$
(so that the actual initial condition can be 0). We run the JPP algorithm as before, using the $\bv{Q}(t)$ values
(not the actual queue values).  However, if we are ever asked to assemble a product, we use \emph{actual} 
raw materials whenever possible.  The only problem comes if we are asked to assemble a product for which
there are not enough actual raw materials available.  However, we know from Theorem \ref{thm:performance1}
that the queue value $Q_m(t)$ never decreases below $\mu_{m,max}$ for any $m\in\{1, \ldots, M\}$. 
It follows that we are \emph{never} asked to use any of our fake raw material.  Therefore, the fake raw
material stays untouched in each queue for all time, and we have: 
\begin{equation} \label{eq:fake-backlog} 
 Q_m(t) = Q_m^{actual}(t) + \mu_{m,max} \: \: \forall m \in \{1, \ldots, M\}, \forall t\geq 0 
 \end{equation} 
The sample path of the system is equivalent to a sample path that does not use fake raw material, but
starts the system in the non-zero initial condition $\bv{Q}(0)$.  Hence, the resulting profit achieved is the
same.  Because the limiting time average profit does not depend on the initial condition, the time average
profit it still at least $\phi^{opt} - B/V$.  However, by (\ref{eq:fake-backlog}) the actual amount of raw material
held is reduced by exactly $\mu_{m,max}$ on each slot, which also reduces the maximum buffer
size $Q_{m,max}$ by exactly this amount. 

\subsection{Demand-Blind Pricing} 

As in \cite{two-prices-allerton07} for the service provider problem, here we consider
the special case when the demand function $F_k(P_k(t), Y(t))$ has the form: 
\begin{equation} \label{eq:demand-blind} 
F_k(P_k(t), Y(t)) = h_k(Y(t))\hat{F}_k(P_k(t)) 
\end{equation} 
for some non-negative functions $h_k(Y(t))$ and $\hat{F}_k(P_k(t))$.  This
holds, for example, when $Y(t)$ represents the integer number of customers at time $t$, 
and $\hat{F}_k(p)$ is the expected demand at price $p$ for each customer, 
so that $F_k(P_k(t),Y(t)) = Y(t)\hat{F}_k(P_k(t))$.  
Under the structure (\ref{eq:demand-blind}),  
the JPP pricing algorithm (\ref{eq:pricing1}) reduces to choosing $P_k(t) \in \script{P}_k$
to maximize: 
\[ V(P_k(t) - \alpha_k)\hat{F}_k(P_k(t)) + \hat{F}_k(P_k(t))\sum_{m=1}^M\beta_{mk}(Q_m(t)-\theta_m) \]
Thus, the pricing can be done without knowledge of the demand state $Y(t)$.

\subsection{Extension to 1-slot Assembly Delay}  \label{subsection:assembly-delay}

Consider now the modified situation where products require one slot for assembly, 
but where consumers still require a product to be provided on the same slot in 
which it is purchased. This can easily be accommodated by  
maintaining an additional set of $K$ \emph{product queues} for
storing finished products.  Specifically, each product queue $k$ is initialized with $D_{k,max}$ 
units of finished products.  The plant also keeps the same material queues $\bv{Q}(t)$ 
as before, and makes all control decisions exactly as before (ignoring the product queues), 
so that every sample path of
queue levels  and control variables 
is the same as before.  
However, when new products are purchased, the consumers do not
wait for assembly, but take the corresponding amount of products out of the product queues. 
This exact amount is replenished when the new products complete their assembly at the 
end of the slot.  Thus, at the beginning of every slot there are always 
exactly $D_{k,max}$ units 
of type $k$ products in the product queues.  The total profit is then the same as before, with the 
exception that the plant incurs a fixed startup cost associated with initializing 
all product queues with a full amount
of finished products. 

\subsection{Extension to Price-Vector Based Demands} 

Suppose that the demand function $F_k(P_k(t), Y(t))$ associated with product $k$ is 
changed to a function $F_k(\bv{P}(t), \bv{Z}(t), Y(t))$  that depends on the full price
vectors $\bv{Z}(t)$ and $\bv{P}(t)$.   This does not significantly 
change the analysis of the Maximum
Profit Theorem (Theorem \ref{thm:max-profit}) or of the dynamic control policy of this 
section.  Specifically, the maximum time average profit $\phi^{opt}$ is still characterized
by randomized $(X,Y)$-only 
algorithms, with the exception that the new price-vector based
demand function 
is used in the optimization of Theorem \ref{thm:max-profit}.  Similarly, the dynamic
control policy of this section 
is only changed by replacing the original demand function 
with the new demand function.  

However, the 2-price result of Theorem \ref{thm:two-price} would no longer apply 
in this setting.  This is because Theorem \ref{thm:two-price} uses 
a strategy of independent pricing
that only applies if the demand function for product $k$ depends only on $P_k(t)$ and
$Y(t)$.  In the case when demands are affected by the full price vector, a modified
analysis can show that $\phi^{opt}$ can be achieved over stationary randomized
algorithms that use at most $\min[K, M]+1$ price vectors $(\bv{Z}^{(i)}, \bv{P}^{(i)})$ for 
each demand state $Y(t) \in \script{Y}$.

\section{Ergodic Models} \label{section:ergodic} 
In this section, we consider the performance of the Joint Purchasing and Pricing Algorithm (JPP) under a more general class of non-i.i.d. consumer demand state $Y(t)$ and material supply state $X(t)$ processes that possess the \emph{decaying memory property} (defined below). In this case, the deterministic queue bounds in 
part (a)  of Theorem \ref{thm:performance1} still hold. This is because part (a) is a
\emph{sample path statement}, which holds under any arbitrary $Y(t)$ and $X(t)$ processes. 
Hence we only have to look at the profit performance. 

\subsection{The Decaying Memory Property}
In this case, we first assume that the stochastic processes $Y(t)$ and $X(t)$ both have well defined time averages. Specifically, we assume that:
\begin{eqnarray}
\lim_{t\rightarrow\infty}\frac{1}{t}\sum_{\tau=0}^{t-1}1\{Y(t)=y\}=\pi(y)\,\,\text{with probability 1}\label{eq:ergodicity_y}\\
\lim_{t\rightarrow\infty}\frac{1}{t}\sum_{\tau=0}^{t-1}1\{X(t)=x\}=\pi(x)\,\,\text{with probability 1},\label{eq:ergodicity_x}
\end{eqnarray}
where $\pi(y)$ and $\pi(x)$ are the same as in the i.i.d. case for all $x$ and $y$. 
Now we consider implementing the $(X, Y)$-only policy in Corollary \ref{cor:1}. Because this policy makes decisions 
every slot purely as a function 
of the current states $X(t)$ and $Y(t)$, and because the limiting fractions of time of 
being in states $x, y$ are the same as in the i.i.d. case, we see that Corollary \ref{cor:1} still 
holds if we take the limit as $t$ goes to infinity, i.e.:
\begin{eqnarray}
\phi^{opt} &=&\lim_{t\rightarrow\infty}\frac{1}{t}\sum_{\tau=0}^{t-1}\expect{\phi^*(\tau)}\label{eq:ergodic_profit}\\
0&=&\lim_{t\rightarrow\infty}\frac{1}{t}\sum_{\tau=0}^{t-1}\expect{A^*_m(\tau)-\sum_{k=1}^K\beta_{mk}Z^*_k(\tau)F_k(P^*_k(\tau), Y(\tau))}\quad\forall m  \label{eq:ergodic_rates}
\end{eqnarray}
where $\expect{\phi^*(\tau)}$, $\expect{A^*_m(\tau)}$ and $\expect{\sum_{k=1}^K\beta_{mk}Z^*_k(\tau)F_k(P^*_k(\tau), Y(\tau))}$ are defined as in Corollary \ref{cor:1}, with expectations taken over the 
distributions of $X(t)$ and $Y(t)$ at time $t$ and the possible randomness of the policy.

We now define $H(t)$ to be the system history up to time slot $t$ as follows:
\begin{eqnarray}
H(t)\triangleq \{X(\tau), Y(\tau)\}_{\tau=0}^{t-1}\cup \{[Q_m(\tau)]_{m=1}^M\}_{\tau=0}^t.\label{eq:history}
\end{eqnarray}
We say that the state processes $X(t)$ and $Y(t)$ have the \emph{decaying memory property} if for any small $\epsilon>0$, there exists a positive integer $T=T_{\epsilon}$, i.e., $T$ is a function of $\epsilon$, such that for any $t_0\in\{0, 1, 2, ...\}$ and any $H(t_0)$, the following holds under the $(X, Y)$-only policy in Corollary \ref{cor:1}: 
\begin{eqnarray}
\bigg|\phi^{opt}-\frac{1}{T}\sum_{\tau=t_0}^{t_0+T-1}\expect{ \phi^*(\tau)\left.|\right. H(t_0)} \bigg|\leq\epsilon.\label{eq:decaying_profit}\\
\bigg| \frac{1}{T}\sum_{\tau=t_0}^{t_0+T-1}\expect{A^*_m(\tau)-\sum_{k=1}^K\beta_{mk}Z^*_k(\tau)F_k(P^*_k(\tau), Y(\tau)) \left.|\right. H(t_0)} \bigg|\leq\epsilon\label{eq:decaying_rate}
\end{eqnarray}
It is easy to see that if $X(t)$ and $Y(t)$ both evolve according to a finite state ergodic Markov chain, then the above will be satisfied. If $X(t)$ and $Y(t)$ are i.i.d. over slots, then $T_{\epsilon} = 1$ for all $\epsilon\geq 0$. 

\subsection{Performance of JPP under the Ergodic Model}
We now present the performance result of JPP under this decaying memory property. 
\begin{thm}\label{theorem:JPP_ergodic}
Suppose the Joint Purchasing and Pricing Algorithm (JPP) is implemented, with $\theta_m$ satisfying condition (\ref{eq:good-theta}),  and that
$\mu_{m, max}\leq Q_m(0)\leq\theta_m+A_{m, max}$ for all $m\in\{1, 2, ..., M\}$. Then the queue backlog values $Q_m(t)$ for all $m\in\{1, \ldots, M\}$ satisfy part (a) of Theorem \ref{thm:performance1}. Further, for any $\epsilon>0$ and $T$ such that (\ref{eq:decaying_profit}) and (\ref{eq:decaying_rate}) hold, we have that:
\begin{eqnarray}
\liminf_{t\rightarrow\infty}\frac{1}{t} \sum_{\tau=0}^{t-1}\expect{\phi_{actual}(\tau) }\geq \phi^{opt}  - \frac{TB}{V}- \epsilon \left(1+\sum_{m=1}^M \frac{\max[\theta_m,A_{m,max}]}{V}\right), \label{eq:exp_Tdrift_final}
\end{eqnarray}
where $B$ is defined in (\ref{eq:B}).
\end{thm}

To understand this result, note that the coefficient multiplying the $\epsilon$ term in the right hand
side of (\ref{eq:exp_Tdrift_final}) is $O(1)$ (recall that $\theta_m$  in (\ref{eq:good-theta}) is linear in $V$). 
Thus, for a given $\epsilon>0$, the final term is $O(\epsilon)$.  Let $T = T_{\epsilon}$ represent the required
mixing time to achieve (\ref{eq:decaying_profit})-(\ref{eq:decaying_rate}) for the given $\epsilon$, and choose
$V = T_{\epsilon}/\epsilon$.  Then by (\ref{eq:exp_Tdrift_final}), we 
are within $\epsilon B + O(\epsilon) = O(\epsilon)$  of the optimal time average profit
$\phi^{opt}$, with buffer size $O(V) = O(T_{\epsilon}/\epsilon)$.  For i.i.d. processes $X(t)$, $Y(t)$, 
we have $T_{\epsilon}=1$
for all $\epsilon\geq 0$, and so the buffer size is $O(1/\epsilon)$.  For processes
$X(t)$, $Y(t)$ that are modulated by finite state ergodic Markov chains, it can be shown that 
$T_{\epsilon} = O(\log(1/\epsilon))$, and so the buffer size requirement is $O((1/\epsilon)\log(1/\epsilon))$ 
\cite{rahul-cognitive-tmc}\cite{two-prices-allerton07}. 

To prove the theorem, it is useful to define the following notions. Using the same Lyapunov function 
in (\ref{eq:lyap-function}) and a positive integer $T$, we define the $T$-slot Lyapunov drift as follows:
\begin{eqnarray}
\Delta_T(H(t))\triangleq \expect{L(\bv{Q}(t+T))-L(\bv{Q}(t))\left.|\right. H(t)},\label{eq:Tslotdrift}
\end{eqnarray}
where $H(t)$ is defined in (\ref{eq:history}) as the past history up to time $t$. It is also useful to define the following notion:
\begin{eqnarray}
\hat{\Delta}_T(t)\triangleq \expect{L(\bv{Q}(t+T))-L(\bv{Q}(t))\left.|\right. \hat{H}_T(t)},\label{eq:Tslotdrift_sp}
\end{eqnarray}
where: 
\begin{eqnarray}
\hat{H}_T(t)= \{H(t)\} \cup \{X(t), Y(t), ..., X(t+T-1), Y(t+T-1)\}\cup\{\bv{Q}(t)\}\label{eq:Tslot_realization_def}
\end{eqnarray}
The value $\hat{H}_T(t)$ represents all history $H(t)$, and additionally includes
the sequence of realizations of the supply and demand states in the interval $\{t, t+1, ..., t+T-1\}$.
It also includes the  backlog vector $\bv{Q}(t)$ (this is already included in $H(t)$, but we explicitly 
include it again in (\ref{eq:Tslot_realization_def}) for convenience).  
Given these values, the expectation in (\ref{eq:Tslotdrift_sp}) 
is with respect to the random demand outcomes $D_k(t)$ and the possibly randomized control
actions. 
We have the following lemma:
\begin{lem}\label{lemma:sample_path_Tdrift}
Suppose the JPP algorithm is implemented with the $\theta_m$ values satisfying (\ref{eq:good-theta}) and $\mu_{m, max}\leq Q_m(0)\leq\theta_m+A_{m, max}$ for all $m\in\{1, 2, ..., M\}$. Then for any $t_0\in\{0, 1, 2, ...\}$, any integers $T$, any $\hat{H}_T(t_0)$, and any $\bv{Q}(t_0)$ value, we have:
\begin{eqnarray}
\hspace{-.2in}&&\hat{\Delta}_T(t_0)-V \sum_{\tau=t_0}^{t_0+T-1}\expect{\phi(\tau)\left.|\right. \hat{H}_T(t_0)}\leq T^2B
-V \sum_{\tau=t_0}^{t_0+T-1}\expect{\phi^*(\tau)\left.|\right. \hat{H}_T(t_0)}\label{eq:Tdrift_samplepath_lemma}\\
\hspace{-.2in}&&  +\sum_{m=1}^M\big(Q_m(t_0)-\theta_m\big)\sum_{\tau=t_0}^{t_0+T-1}\expect{A^*_m(\tau)-\sum_{k=1}^K\beta_{mk}Z^*_k(\tau)F_k(P^*_k(\tau), Y(\tau))\left.|\right. \hat{H}_T(t_0)},\nonumber
\end{eqnarray}
with $B$ defined in (\ref{eq:B}) 
and $\phi^*(\tau), A^*_m(\tau)$, $Z_k^*(\tau)$ and $P^*_k(\tau)$ are variables generated by any other policy
that can be implemented over the $T$ slot interval (including those that know the future $X(\tau)$, $Y(\tau)$
states in this
interval).  
\end{lem}
\begin{proof}
See Appendix C. 
\end{proof}

We now prove Theorem  \ref{theorem:JPP_ergodic}. 
\begin{proof} (Theorem \ref{theorem:JPP_ergodic})
Fix any $t_0\geq0$. Taking expectations 
on both sides of (\ref{eq:Tdrift_samplepath_lemma}) (conditioning on the information $H(t_0)$ that is already included in $\hat{H}_T(t_0)$)  yields:
\begin{eqnarray*}
\hspace{-.2in}&&\Delta_T(H(t_0)) - V \sum_{\tau=t_0}^{t_0+T-1}\expect{\phi(\tau)\left.|\right. H(t_0)}\leq T^2B- V\sum_{\tau=t_0}^{t_0+T-1}\expect{\phi^*(\tau)\left.|\right. H(t_0)} \\
\hspace{-.2in}&&+\sum_{m=1}^M(Q_m(t_0)-\theta_m)\sum_{\tau=t_0}^{t_0+T-1}\expect{A^*_m(\tau)-\sum_{k=1}^K\beta_{mk}Z^*_k(\tau)F_k(P^*_k(\tau), Y(\tau))\left.|\right. H(t_0)}. 
\end{eqnarray*}
Now plugging in the policy in Corollary \ref{cor:1} and using the $\epsilon$ and $T$ that
 yield  (\ref{eq:decaying_profit}) and (\ref{eq:decaying_rate}) in the above, 
and using the fact that $|Q_m(t_0)-\theta_m|\leq \max[\theta_m, A_{m,max}]$, we have:
\begin{eqnarray}
\Delta_T(H(t_0)) - V \sum_{\tau=t_0}^{t_0+T-1}\expect{\phi(\tau)\left.|\right. H(t_0)}\leq \nonumber \\
T^2B- VT\phi^{opt} + VT\epsilon 
+T\sum_{m=1}^M \max[\theta_m, A_{m,max}]\epsilon.\label{eq:exp_Tdrift_0}
\end{eqnarray}
We can now take expectations of (\ref{eq:exp_Tdrift_0}) over $H(t_0)$ to obtain:
\begin{eqnarray}
\expect{L(\bv{Q}(t_0+T)) - L(\bv{Q}(t_0))} - V \sum_{\tau=t_0}^{t_0+T-1}\expect{\phi(\tau) }\leq \nonumber \\
T^2B- VT\phi^{opt} + VT\epsilon 
+T\sum_{m=1}^M \max[\theta_m, \mu_{m,max}]\epsilon.   \nonumber
\end{eqnarray}
Summing the above over $t_0= 0, T, 2T, ..., (J-1)T$ for some positive $J$ and dividing both sides by $VTJ$, we get:
\begin{eqnarray}
\frac{\expect{L(\bv{Q}(JT)) - L(\bv{Q}(0))}}{VTJ} - \frac{1}{JT} \sum_{\tau=0}^{JT-1}\expect{\phi(\tau) }\leq  \nonumber \\
\frac{TB}{V}- \phi^{opt} + \epsilon 
+\sum_{m=1}^M \frac{\max[\theta_m, A_{m,max}]\epsilon}{V}.\label{eq:exp_Tdrift_2}
\end{eqnarray}
By rearranging terms and using the fact that $L(t)\geq0$ for all $t$, we obtain:
\begin{eqnarray*}
\frac{1}{JT} \sum_{\tau=0}^{JT-1}\expect{\phi(\tau) }\geq \phi^{opt}  - \frac{TB}{V}- \epsilon \left(1+\sum_{m=1}^M \frac{\max[\theta_m,A_{m,max}]}{V}\right) -  \frac{\expect{ L(\bv{Q}(0))}}{VTJ}.
\end{eqnarray*}
Taking the liminf as $J\rightarrow\infty$, we have:
\begin{eqnarray}
\liminf_{t\rightarrow\infty}\frac{1}{t} \sum_{\tau=0}^{t-1}\expect{\phi(\tau) }\geq \phi^{opt}  - \frac{TB}{V}- \epsilon \left(1+\sum_{m=1}^M \frac{\max[\theta_m,A_{m,max}]}{V}\right).
\end{eqnarray}
This completes the proof of Theorem \ref{theorem:JPP_ergodic}.
\end{proof}

\section{Arbitrary Supply and Demand Processes} \label{section:non-ergodic} 
In this section, we further relax our assumption about the supply and demand processes to allow for  \emph{arbitrary} $X(t)$ and $Y(t)$ processes, and we look at the performance of the JPP algorithm. Note that in this case, the notion of ``optimal'' time average profit may no longer be applicable. Thus, we instead compare the performance of JPP with the optimal value that one can achieve over a time interval of length $T$. This optimal value is defined to be the supremum of the achievable average profit over any policy, including those which  know the entire realizations of $X(t)$ and $Y(t)$  over the $T$ slots at the very beginning of the interval. We will call such a policy a \emph{$T$-slot Lookahead policy} in the following. We will show that in this case, JPP's performance is close to that under an optimal $T$-slot Lookahead policy.  This $T$-slot lookahead
metric is similar to the one used in 
\cite{neely-stock-arxiv}\cite{neely-universal-scheduling}, with the exception that here we
compare to policies that
know only the $X(t)$ and $Y(t)$ realizations and not the demand $D_k(t)$ realizations. 

\subsection{The $T$-slot Lookahead Performance}
Let $T$ be a positive integer and let $t_0\geq0$. 
Define $\phi_T(t_0)$ as the optimal expected
profit achievable over the interval $\{t_0, t_0+1, ..., t_0+T-1\}$ by any policy that has the complete knowledge 
of the entire $X(t)$ and $Y(t)$ processes over this interval and that ensures that the quantity of the raw materials purchased over the interval is equal to 
 the expected amount consumed.  Note here that although the future $X(t)$, $Y(t)$
values are assumed to be known, the random demands $D_k(t)$ are still unknown. 
Mathematically, $\phi_T(t_0)$ can be defined as the solution to the following optimization problem:
\begin{eqnarray}
\hspace{-.3in}(\textbf{P1}) &&\nonumber\\
\hspace{-.3in}  \max: &&\phi_T(t_0)=\sum_{\tau=t_0}^{t_0+T-1} \expect{\phi(\tau)\left.|\right. \hat{H}_T(t_0)}\label{eq:universal_obj}\\
\hspace{-.3in}s.t. && \sum_{\tau=t_0}^{t_0+T-1}  \expect{A_m(\tau)-\sum_{k=1}^K\beta_{mk}Z_k(\tau)F_k(P_k(\tau), Y(\tau))\left.|\right. \hat{H}_T(t_0)} = 0, \,\,\,\forall m \label{eq:universal_cond}\\
\hspace{-.3in}&& \text{Constraints}\,\, (\ref{eq:p-k-constraint}), (\ref{eq:Amax}), (\ref{eq:integer-a}), (\ref{eq:cmax} ).\label{eq:universal_cond2}
\end{eqnarray}
Here $\hat{H}_T(t_0)$ is defined in (\ref{eq:Tslot_realization_def}) and includes 
the sequence of realizations of $X(t)$ and $Y(t)$ during the interval $\{t_0, ..., t_0+T-1\}$; 
$\phi(\tau)$ is defined in Equation (10) as the instantaneous profit obtained at time $\tau$; $A_m(\tau)$ and $\sum_{k=1}^K\beta_{mk}Z_k(\tau)F_k(P_k(\tau), Y(\tau))$ are the number of newly ordered parts and the expected
number of consumed parts in time $\tau$, respectively; and the expectation is taken over the randomness in $D_k(\tau)$, due to the fact that the demand at time $\tau$ is a random variable. We note that in Constraint (\ref{eq:universal_cond}), we actually do not require that in every intermediate step the raw material queues have enough for production. This means that a $T$-slot Lookahead policy can make products out of its future raw materials, provided that they are purchased before the interval ends. 
Since purchasing no materials and selling no products over the entire interval is a valid policy, we see that the value $\phi_T(t_0)\geq0$ for all $t_0$ and all $T$. 

In the following, we will look at the performance of JPP over the interval from $0$ to $JT-1$, which is divided into a total of $J$ frames with length $T$ each. We show that for any $J>0$, the JPP algorithm yields an average profit over $\{0, 1, ..., JT-1\}$ that is close to the profit obtained with an optimal $T$-slot Lookahead policy implemented on each $T$-slot frame.

\subsection{Performance of JPP under arbitrary supply and demand}
The following theorem summarizes the results. 
\begin{thm}\label{theorem:JPP_arbitrary}
Suppose the Joint Purchasing and Pricing Algorithm (JPP) is implemented, with $\theta_m$ satisfying condition (\ref{eq:good-theta}) and that $\mu_{m, max}\leq Q_m(0)\leq\theta_m+A_{m, max}$ for all $m\in\{1, 2, ..., M\}$. Then for any arbitrary $X(t)$ and $Y(t)$ processes, the queue backlog values satisfies part (a) of Theorem \ref{thm:performance1}. Moreover, for any positive integers $J$ and $T$, and any $\hat{H}_{JT}(0)$ (which specifies
the initial queue vector $\bv{Q}(0)$ according to the above bounds, and specifies
all $X(\tau)$, $Y(\tau)$ values for $\tau \in\{0, 1, \ldots, JT-1\}$), 
the time average profit over the interval $\{0, 1, ..., JT-1\}$ satisfies:
\begin{eqnarray*}
\frac{1}{JT} \sum_{\tau=0}^{JT-1}\expect{\phi(\tau)\left.|\right. \hat{H}_{JT}(0)}\geq\frac{1}{JT}\sum_{j=0}^{J-1}\phi_T(jT) - \frac{BT}{V} -\frac{L(\bv{Q}(0))}{ VJT}. 
\end{eqnarray*}
where $\phi_T(jT)$ is defined to be the optimal value of the problem (\textbf{P1}) over the interval $\{jT, ..., (j+1)T-1\}$. The constant $B$ is defined in (\ref{eq:B}). 
\end{thm}
\begin{proof} (Theorem \ref{theorem:JPP_arbitrary}) 
Fix any $t_0\geq0$. We denote the optimal solution to the problem (\textbf{P1}) over the interval $\{t_0, t_0+1, ..., t_0+T-1\}$ as:
\[\{\phi^*(\tau), A^*_m(\tau), Z^*_k(\tau), P^*_k(\tau)\}_{\tau=t_0, ..., t_0+T-1}^{m=1, ..., M}.\]
Now recall (\ref{eq:Tdrift_samplepath_lemma}) as follows:
\begin{eqnarray}
\hspace{-.2in}&&\hat{\Delta}_T(t_0)-V \sum_{\tau=t_0}^{t_0+T-1}\expect{\phi(\tau)\left.|\right. \hat{H}_T(t_0)}\leq T^2B-V \sum_{\tau=t_0}^{t_0+T-1}\expect{\phi^*(\tau)\left.|\right. \hat{H}_T(t_0)}\label{eq:Tdrift_samplepath_lemma_reuse}\\
\hspace{-.2in}&&  +\sum_{m=1}^M\big(Q_m(t_0)-\theta_m\big)\sum_{\tau=t_0}^{t_0+T-1}\expect{A^*_m(\tau)-\sum_{k=1}^K\beta_{mk}Z^*_k(\tau)F_k(P^*_k(\tau), Y(\tau))\left.|\right. \hat{H}_T(t_0)},\nonumber
\end{eqnarray}
 Now note that the actions $\phi^*(\tau)$, $A^*_m(\tau)$ $Z_k^*(\tau)$ and $P^*_k(\tau)$ satisfy 
 (\ref{eq:universal_obj})-(\ref{eq:universal_cond}) and so: 
\begin{eqnarray}
\hat{\Delta}_T(t_0)-V \sum_{\tau=t_0}^{t_0+T-1}\expect{\phi(\tau)\left.|\right. \hat{H}_T(t_0)}\leq T^2B-V \phi_T(t_0).  \label{eq:almost-done} 
\end{eqnarray}
Now note that since JPP makes actions based only on the current queue backlog and $X(\tau)$, $Y(\tau)$ states, 
we have: 
\begin{eqnarray*}
 \hat{\Delta}_T(t_0) &=& \expect{L(\bv{Q}(t_0+T))-L(\bv{Q}(t_0))|\hat{H}_{JT}(0), \bv{Q}(t_0)} \\
 \sum_{\tau=t_0}^{t_0+T-1} \expect{\phi(\tau)|\hat{H}_T(t_0)} &=& \sum_{\tau=t_0}^{t_0+T-1}\expect{\phi(\tau)|\hat{H}_{JT}(0), \bv{Q}(t_0)} 
 \end{eqnarray*}
 That is, conditioning on the additional $X(\tau)$, $Y(\tau)$ states for $\tau$ outside of the $T$-slot interval
 does not change the expectations.  Using these in (\ref{eq:almost-done}) yields: 
 \begin{eqnarray*}
\expect{L(\bv{Q}(t_0+T))-L(\bv{Q}(t_0))|\hat{H}_{JT}(0), \bv{Q}(t_0)} 
-V\sum_{\tau=t_0}^{t_0+T-1}\expect{\phi(\tau)|\hat{H}_{JT}(0), \bv{Q}(t_0)} 
\leq \\
T^2B-V \phi_T(t_0).  
\end{eqnarray*}
Taking expectations of the above with respect to the random $\bv{Q}(t_0)$ states (given $\hat{H}_{JT}(0)$)
yields: 
 \begin{eqnarray*}
\expect{L(\bv{Q}(t_0+T))-L(\bv{Q}(t_0))|\hat{H}_{JT}(0)} 
-V\sum_{\tau=t_0}^{t_0+T-1}\expect{\phi(\tau)|\hat{H}_{JT}(0)} 
\leq \\
T^2B-V \phi_T(t_0).  
\end{eqnarray*}
 
Letting $t_0=jT$ and summing over $j=0, 1, ... J-1$ yields:
\begin{eqnarray*}
\expect{L(\bv{Q}(JT))-L(\bv{Q}(0))\left.|\right. \hat{H}_{JT}(0)} - V \sum_{\tau=0}^{JT-1}\expect{\phi(\tau)\left.|\right. \hat{H}_{JT}(0)}\\
\leq JT^2B-V \sum_{j=0}^{J-1}\phi_T(jT). 
\end{eqnarray*}
Rearranging the terms, dividing both sides by $VJT$ and using the fact that $L(t)\geq0$ for all $t$, we get:
\begin{eqnarray*}
\frac{1}{JT} \sum_{\tau=0}^{JT-1}\expect{\phi(\tau)\left.|\right. \hat{H}_{JT}(0)}\geq\frac{1}{JT}\sum_{j=0}^{J-1}\phi_T(jT) - \frac{BT}{V} -\frac{\expect{L(\bv{Q}(0))\left.|\right. \hat{H}_{JT}(0)}}{ VJT}. 
\end{eqnarray*}
Because $\bv{Q}(0)$ is included in the $\hat{H}_{JT}(0)$ information, we have $\expect{L(\bv{Q}(0))|\hat{H}_{JT}(0)} = L(\bv{Q}(0))$. 
This proves Theorem \ref{theorem:JPP_arbitrary}. 
\end{proof}

\section{Conclusions} 

We have developed a dynamic pricing and purchasing strategy that achieves time
average profit that is arbitrarily close to optimal, with a corresponding tradeoff in the maximum
buffer size required for the raw material queues.   When the supply and demand states $X(t)$ and $Y(t)$
are i.i.d. over slots, we showed that the profit is within $O(1/V)$ of optimality, with a worst-case
buffer requirement of $O(V)$, where $V$ is a parameter that can be chosen as desired to affect the 
tradeoff.  Similar performance was shown for ergodic $X(t)$ and $Y(t)$ processes with a mild
decaying memory property, where the deviation from optimality also depends on a ``mixing time''
parameter.   Finally, we showed that the same algorithm provides efficient performance for \emph{arbitrary} 
(possibly non-ergodic) $X(t)$ and $Y(t)$ processes.  In this case, efficiency is measured
against an ideal $T$-slot lookahead policy with knowledge of the future $X(t)$ and $Y(t)$ values up
to $T$ slots.

Our Joint Purchasing and Pricing (JPP) algorithm reacts to the observed system
state on every slot, and does not require knowledge of the probabilities associated with future
states.   Our analysis technique is based on Lyapunov optimization, and uses
a Lyapunov function that ensures enough inventory is available to take advantage of emerging
favorable demand states.  This analysis approach can be applied to very large systems, without the 
curse of dimensionality issues seen by other approaches such as dynamic programming.

\section*{Appendix A --- Proof of  Necessity for Theorem \ref{thm:max-profit}}

\begin{proof} (Necessity portion of Theorem \ref{thm:max-profit}) 
For simplicity, we assume the system is initially empty. 
Because $X(t)$ and $Y(t)$ are stationary, we have
$Pr[X(t) = x]  = \pi(x)$ and $Pr[Y(t) = y]  = \pi(y)$ for all $t$ and all $x \in \script{X}$, 
$y\in\script{Y}$. 
Consider any algorithm that makes decisions for $\bv{A}(t), \bv{Z}(t), \bv{P}(t)$ over
time, and also makes decisions for $\tilde{\bv{D}}(t)$ 
according to the scheduling 
constraints (\ref{eq:scheduling-constraints1})-(\ref{eq:scheduling-constraints2}). 
Let $\phi_{actual}(t)$ represent the  actual instantaneous profit associated with this
algorithm. 
 
 Define $\overline{\phi}_{actual}$ as the $\limsup$ time average expectation of $\phi_{actual}(t)$, and 
 let $\{\tilde{t}_i\}$ represent the subsequence of times over which the $\limsup$ is achieved, so that: 
 \begin{equation} \label{eq:liminf-appa} 
  \lim_{i \rightarrow \infty} \frac{1}{\tilde{t}_i} \sum_{\tau=0}^{\tilde{t}_i-1} \expect{\phi_{actual}(\tau)} = \overline{\phi}_{actual} 
  \end{equation} 
 Let $\overline{c}(t), \overline{r}(t)$, $\overline{a}_m(t), \overline{\mu}_m(t)$ represent the following
  time averages up to slot $t$: 
 \begin{eqnarray}
 \overline{c}(t) &\defequiv& \frac{1}{t}\sum_{\tau=0}^{t-1} \expect{c(\bv{A}(\tau), X(\tau))} \label{eq:appa1} \\
 \overline{r}(t) &\defequiv& \frac{1}{t}\sum_{\tau=0}^{t-1} \sum_{k=1}^K\expect{Z_k(\tau)\tilde{D}_k(\tau)(P_k(\tau)-\alpha_k)} \label{eq:appa2} \\
 \overline{a}_m(t) &\defequiv& \frac{1}{t}\sum_{\tau=0}^{t-1} \expect{A_m(\tau)} \label{eq:appa3} \\
 \overline{\mu}_m(t) &\defequiv& \frac{1}{t}\sum_{\tau=0}^{t-1} \sum_{k=1}^K \beta_{mk}\expect{Z_k(\tau)\tilde{D}_k(\tau)}  \label{eq:appa4} 
 \end{eqnarray}
 Because the system is initially empty, we cannot use more raw materials of type $m$ 
 up to time $t$ than
 we have purchased,  and hence: 
 \begin{equation} \label{eq:supply-constraint-appa}
  \overline{a}_m(t) \geq \overline{\mu}_m(t) \: \: \mbox{ for all $t$ and all $m \in \{1, \ldots, M\}$} 
  \end{equation} 
   Further, note that the sum profit up to time $t$ is given by: 
  \begin{equation} \label{eq:sum-profit-appa} 
  \frac{1}{t}\sum_{\tau=0}^{t-1} \expect{\phi_{actual}(\tau)} = -\overline{c}(t) + \overline{r}(t)
  \end{equation} 
 We now have the following claim, proven at the end of this section. 
 
 \emph{Claim 1:}  For each slot $t$ and each $x\in\script{X}$, $y\in\script{Y}$, 
 $\bv{a} \in \script{A}(x)$, $k \in \{1, \ldots, K\}$, there are
 functions $\theta(\bv{a}, x, t)$, $\nu_k(y,t)$, and $d_k(y,t)$ 
 such that: 
 \begin{eqnarray}
 \overline{c}(t) &=& \sum_{x \in \script{X}} \pi(x)\sum_{\bv{a} \in \script{A}(x)} \theta(\bv{a}, x, t)c(\bv{a}, x) \label{eq:c1-appa} \\
  \overline{a}_m(t) &=& \sum_{x \in \script{X}} \pi(x) \sum_{\bv{a}\in\script{A}(x)} \theta(\bv{a}, x, t)a_m \label{eq:c2-appa} \\
  \overline{r}(t) &=& \sum_{y \in \script{Y}}\pi(y)\sum_{k=1}^K  \nu_k(y,t) \label{eq:c3-appa}\\
  \overline{\mu}_m(t) &=& \sum_{y\in\script{Y}} \pi(y)\sum_{k=1}^K \beta_{mk} d_k(y,t) \label{eq:c4-appa}
 \end{eqnarray}
 and such that: 
 \begin{eqnarray}
 0 \leq \theta(\bv{a}, x, t) \leq 1 \: \: , \: \: \sum_{\bv{a} \in \script{A}(x)} \theta(\bv{a}, x, t) = 1 \: \: \forall x\in\script{X}, \bv{a} \in \script{A}(x) \label{eq:compact1} 
 \end{eqnarray}
 and such that  for each $y \in \script{Y}$, $k \in \{1, \ldots, K\}$,  
 the vector $(\nu_k(y,t); d_k(y,t))$ is in the convex hull of the following
  two-dimensional compact set: 
 \begin{equation} \label{eq:compact2} 
  \Omega_k(y) \defequiv \{ (\nu, \mu) \left|\right. \nu = (p -\alpha_k)zF_k(p, y) \: , \:  \mu = zF_k(p, y) \: , \:  p \in \script{P} \: , \: z \in \{0, 1\}  \}.  
  \end{equation} 
 
 The values $[\theta(\bv{a}, x, t); (\nu_k(y,t); d_k(y,t))]_{x \in \script{X}, y\in\script{Y}, \bv{a} \in\script{A}(x)}$ of 
 Claim 1 can be 
 viewed  as a finite or countably infinite  
 dimensional vector sequence indexed by time 
 $t$ that is contained in the compact set  defined by (\ref{eq:compact1}) and the convex
 hull of (\ref{eq:compact2}).  Hence, by a classical diagonalization procedure, 
 every infinite sequence contains a
 convergent subsequence that is contained in the set \cite{billingsley}.  
 
 Consider the infinite sequence of times $\tilde{t}_i$ (for which (\ref{eq:liminf-appa}) holds), 
 and let $t_i$ represent the infinite subsequence for which $\theta(\bv{a}, x, t_i)$, 
 $\nu_k(y, t_i)$, and $d_k(y, t_i)$ converge.  Let $\theta(\bv{a}, x)$, $\nu_k(y)$, 
 and $d_k(y)$ represent the limiting values.  Define $\hat{c}$, $\hat{a}_m$,
 $\hat{r}$, and $\hat{\mu}_m$ as the corresponding limiting values of (\ref{eq:c1-appa})-(\ref{eq:c4-appa}).   
 \begin{eqnarray*}
 \hat{c} &=& \sum_{x \in \script{X}} \pi(x)\sum_{\bv{a} \in \script{A}(x)} \theta(\bv{a}, x)c(\bv{a}, x) \\
  \hat{a}_m &=& \sum_{x \in \script{X}} \pi(x) \sum_{\bv{a}\in\script{A}(x)} \theta(\bv{a}, x)a_m \\
  \hat{r} &=& \sum_{y \in \script{Y}}\pi(y)\sum_{k=1}^K \nu_k(y) \\
  \hat{\mu}_m &=& \sum_{y\in\script{Y}} \pi(y)\sum_{k=1}^K \beta_{mk} d_k(y)
 \end{eqnarray*}
 Further, the limiting values of $\theta(\bv{a}, x)$ retain the properties (\ref{eq:compact1}) and hence
 can be viewed as probabilities.  Furthermore, taking limits as $t_i\rightarrow\infty$ in (\ref{eq:supply-constraint-appa}) and (\ref{eq:sum-profit-appa}) yields: 
 \begin{eqnarray*}
& \hat{a}_m \geq \hat{\mu}_m \: \: \mbox{ for all $m \in \{1, \ldots, M\}$}&  \\
 &\overline{\phi}_{actual} = -\hat{c} + \hat{r}& 
 \end{eqnarray*}
 Finally, note that for each $k \in \{1, \ldots, K\}$ and each $y \in \script{Y}$, the 
 vector $(\nu_k(y), d_k(y))$ is in the convex hull of the set (\ref{eq:compact2}), and 
 hence can be achieved by an $(X,Y)$-only policy that chooses $\bv{Z}(t)$ and 
 $\bv{P}(t)$ as a random
 function of the observed value of $Y(t)$, such that $Z_k(t) \in \{0, 1\}$ and $P_k(t) \in\script{P}_k$ for 
 all $k$, and: 
 \begin{eqnarray*}
 \nu_k(y) &=& \expect{(P_k(t)-\alpha_k)Z_k(t)F_k(P_k(t), y)\left|\right. Y(t) = y} \\
 d_k(y) &=& \expect{Z_k(t)F_k(P_k(t), y) \left|\right. Y(t) = y }
 \end{eqnarray*}
 It follows that $\overline{\phi}_{actual}$ is an achievable value of $\phi$  for which there are 
 appropriate auxiliary variables  that satisfy the constraints of the 
 optimization problem of Theorem \ref{thm:max-profit}.  However, $\phi^{opt}$ is defined
 as the supremum over all such $\phi$ values, and hence we must 
 have $\overline{\phi}_{actual} \leq \phi^{opt}$.
 \end{proof}
 
 It remains only to prove Claim 1. 

 \begin{proof} (Claim 1)  We  can re-write the expression for  $\overline{c}(t)$ in (\ref{eq:appa1}) 
 as follows: 
\begin{eqnarray*}
\overline{c}(t) &=& \frac{1}{t}\sum_{\tau=0}^{t-1}\sum_{x \in \script{X}} \sum_{\bv{a}}c(\bv{a}, x)\pi(x)Pr[\bv{A}(\tau)=\bv{a}\left|\right.X(\tau)=x]  \\
&=& \sum_{x \in \script{X}}\pi(x)\sum_{\bv{a}} \theta(\bv{a}, x, t)c(\bv{a}, x)
\end{eqnarray*}
where $\theta(\bv{a}, x, t)$ is defined: 
\[ \theta(\bv{a}, x, t) \defequiv \frac{1}{t}\sum_{\tau=0}^{t-1} Pr[\bv{A}(\tau) = \bv{a}\left|\right. X(\tau) = x] \]
and satisfies: 
\[ 0 \leq \theta(\bv{a}, x, t) \leq 1 \: \: , \: \: \sum_{\bv{a}}\theta(\bv{a}, x, t) = 1 \: \: \forall t, x \in \script{X} \]
This proves (\ref{eq:c1-appa}). 
Likewise, we can rewrite the expression for $\overline{a}_m(t)$ in (\ref{eq:appa3}) as follows: 
\begin{eqnarray*}
\overline{a}_m(t) &=& \sum_{x \in \script{X}} \pi(x) \sum_{\bv{a}} \theta(\bv{a}, x, t)a_m
\end{eqnarray*}
This proves (\ref{eq:c2-appa}). 

To prove (\ref{eq:c3-appa})-(\ref{eq:c4-appa}), note that 
 we can rewrite the expression for $\overline{r}(t)$ in (\ref{eq:appa2})  as follows: 
\begin{eqnarray}
\overline{r}(t) &=& \frac{1}{t}\sum_{\tau=0}^{t-1} \sum_{y \in \script{Y}}\sum_{k=1}^K \pi(y)\expect{(P_k(\tau)-\alpha_k)Z_k(\tau)\tilde{D}_k(\tau)\left|\right. Y(\tau) = y}
 \label{eq:baz0-appa} 
\end{eqnarray}
Now for all $y \in \script{Y}$,  all  vectors $\bv{z} \in \{0,1\}^K$, $\bv{p} \in \script{P}^K$, 
and all slots $t$, 
define $\gamma_k(y, \bv{z}, \bv{p}, t)$ as follows: 
\begin{eqnarray*}
\gamma_k(y, \bv{z}, \bv{p}, t) \defequiv  \left\{ \begin{array}{ll}
                             \frac{\expect{z_k\tilde{D}_k(t)\left|\right. Y(t) = y, \bv{Z}(t) = \bv{z}, \bv{P}(t) = \bv{p}}}{\expect{z_kD_k(t)\left|\right. Y(t) = y, \bv{Z}(t) = \bv{z}, \bv{P}(t) = \bv{p}}}  & \mbox{ if the denominator is non-zero}  \\
0 & \mbox{ otherwise} 
                            \end{array}
                                 \right.
\end{eqnarray*}
Note that $\tilde{D}_k(t) \leq D_k(t)$, and so $0 \leq \gamma_k(y, \bv{z}, \bv{p}, t) \leq 1$.
It follows by definition that: 
\[ \expect{z_k\tilde{D}_k(t)\left|\right.Y(t) = y, \bv{Z}(t) = \bv{z}, \bv{P}(t) = \bv{p}} = \gamma_k(y, \bv{z}, \bv{p}, t)F_k(p_k, y) \]
 Using the above equality  with iterated expectations in 
(\ref{eq:baz0-appa}) yields: 
\begin{eqnarray}
\hspace{-.3in}\overline{r}(t) &=& \sum_{y \in \script{Y}}  \sum_{k=1}^K   \pi(y) \times \nonumber \\
&&  \frac{1}{t}\sum_{\tau=0}^{t-1} \expect{(P_k(\tau)-\alpha_k)
     \gamma_k(y, \bv{Z}(\tau), \bv{P}(\tau), \tau)F_k(P_k(\tau), y) 
\left|\right. Y(\tau) = y}  \label{eq:r-appa}
\end{eqnarray} 
With similar analysis, we can rewrite the expression for $\overline{\mu}_m(t)$ in  (\ref{eq:appa4}) as follows: 
\begin{eqnarray}
\hspace{-.3in} \overline{\mu}_m(t) &=& \sum_{y\in\script{Y}} \sum_{k=1}^K\beta_{mk}\pi(y) \times \nonumber \\
&& \frac{1}{t}\sum_{\tau=0}^{t-1} \expect{   \gamma_k(y, \bv{Z}(\tau), \bv{P}(\tau), \tau)F_k(P_k(\tau), y)          \left|\right. Y(t) = y} \label{eq:d-appa} 
\end{eqnarray}
Now define $\nu_k(y,t)$ and $d_k(y,t)$ as the corresponding time average 
 expectations inside the 
summation terms of (\ref{eq:r-appa}) and (\ref{eq:d-appa}), respectively, so that:
\begin{eqnarray*}
\overline{r}(t) &=& \sum_{y\in\script{Y}}\pi(y)\sum_{k=1}^K  \nu_k(y, t) \\
\overline{\mu}_m(t) &=& \sum_{y\in\script{Y}}\pi(y)\sum_{k=1}^K \beta_{mk}d_k(y,t)
\end{eqnarray*}
Note that the time average expectation  over $t$ slots used in the 
definitions of $\nu_k(y,t)$ and $d_k(y,t)$ can be viewed as an operator that produces
a convex combination.  Specifically, the two-dimensional 
vector $(\nu_k(y,t); d_k(y,t))$ can be viewed as an element of the convex hull
of the following set $\hat{\Omega}_k(y)$:
\[ \hat{\Omega}_k(y) \defequiv \{ (\nu, d) \left|\right. \nu = (p-\alpha_k)\gamma F_k(p, y) \: , \:  d = \gamma F_k(p, y) \: , \:  p \in \script{P} \: , \: 0 \leq \gamma \leq 1   \}   \]
However, it is not difficult to show that the convex hull of the set $\hat{\Omega}_k(y)$ is the 
same as the convex hull of the set $\Omega_k(y)$ defined in (\ref{eq:compact2}).\footnote{This 
can be seen by noting that $\Omega_k(y) \subset \hat{\Omega}_k(y) \subset Conv(\Omega_k(y))$
and then taking convex hulls of this inclusion relation.} 
This proves (\ref{eq:c3-appa}) and (\ref{eq:c4-appa}). 
\end{proof} 

\section*{Appendix B --- Proof of the 2-Price Theorem (Theorem \ref{thm:two-price})} 
\begin{proof} (Theorem \ref{thm:two-price})  The 
 proof follows the work of \cite{two-prices-allerton07}. 
 For each product $k \in \{1, \ldots, K\}$
 and each possible demand state $y \in \script{Y}$, 
  define constants $\hat{r}_k(y)$ and $\hat{d}_{k}(y)$ as follows: 
 \begin{eqnarray*}
 \hat{r}_k(y) &\defequiv& \expect{Z_k(t)(P_k(t)-\alpha_k)F_k(P_k(t), y)\left|\right. Y(t) = y}  \\
 \hat{d}_{k}(y) &\defequiv& \expect{Z_k(t)F_k(P_k(t), y)\left|\right. Y(t) = y} 
 \end{eqnarray*}
 where $\bv{Z}(t) = (Z_1(t), \ldots, Z_K(t))$ and $\bv{P}(t) = (P_1(t), \ldots, P_K(t))$ are 
 the stationary randomized decisions given in the statement Theorem \ref{thm:two-price}. 
Thus, by (\ref{eq:two-price1}) and (\ref{eq:two-price2}): 
\begin{eqnarray}
&\sum_{k=1}^K \sum_{y\in\script{Y}} \pi(y)\hat{r}_k(y) \geq \hat{r}& \label{eq:baz1} \\
&\sum_{k=1}^K \beta_{mk} \sum_{y \in \script{Y}} \pi(y) \hat{d}_{mk}(y) \leq \hat{\mu}_m \: \: \mbox{for all $m \in \{1, \ldots, M\}$}& \label{eq:baz2} 
\end{eqnarray}
Now consider a particular  $k, y$, and 
define the two-dimensional  set $\Omega(k, y)$ as follows: 
\[ \Omega(k,  y) = \{ (r;  d) \left|\right. r = z(p-\alpha_k)F_k(p, y), d = zF_k(p, y), p \in \script{P}, z \in \{0, 1\}\} \]
We now use the fact that if $\bv{C}$ is any general random vector that takes values in a general 
set $\script{C}$, then 
$\expect{\bv{C}}$ is in the convex hull of $\script{C}$ \cite{now}.  Note that for any random 
choice of $Z_k(t) \in \{0, 1\}, P_k(t) \in \script{P}_k$, we have: 
\[ (Z_k(t)(P_k(t) - \alpha_k)F_k(P_k(t), y); Z_k(t) F_k(P_k(t), y)) \in \Omega(k,y) \]
Hence, the conditional expectation of this random vector given $Y(t) = y$, given by 
$(\hat{r}_k(y); \hat{d}_{k}(y))$,  is in the convex hull of $\Omega(k, y)$.  Because $\Omega(k,y)$ is a 
two dimensional set, any element of its convex hull can be expressed as a convex combination that
uses at most three elements  of $\Omega(k,y)$ (by Caratheodory's Theorem \cite{bertsekas-convex}).  Moreover, because
the set $\script{P}$ is compact and $F_k(p, y)$ is a continuous function of 
$p \in \script{P}$ for each $y \in \script{Y}$, the set $\Omega(k, y)$ is compact and
hence any point on the \emph{boundary} of its convex hull  
can be described by a convex combination of at most \emph{two} elements of $\Omega(k,y)$
(see, for example,  \cite{two-prices-allerton07}).  Let $(\hat{r}_k^*(y), \hat{d}_{k}(y))$ be
the boundary point with the largest value of the first entry given the second entry is 
$\hat{d}_{k}(y)$.  We thus have $\hat{r}_k^*(y) \geq \hat{r}_k(y)$, and writing the 
convex combination with two elements we have: 
\begin{eqnarray*}
 (\hat{r}_k^*(y); \hat{d}_{k}(y)) &=&  \eta^{(1)}\left(z^{(1)}(p^{(1)}-\alpha_k)F_k(p^{(1)}, y) ; 
z^{(1)}F_k(p^{(1)}, y)\right) \\
&&  + \eta^{(2)}\left(z^{(2)}(p^{(2)}-\alpha_k)F_k(p^{(2)}, y) ; 
z^{(2)}F_k(p^{(2)}, y)\right)
\end{eqnarray*}
for some set of decisions  $(z^{(1)}, p^{(1)})$ and $(z^{(2)}, p^{(2)})$ (with 
$z^{(i)} \in \{0, 1\}, p^{(i)} \in \script{P}$) and
probabilities $\eta^{(1)}$ and $\eta^{(2)}$ such that $\eta^{(1)} + \eta^{(2)} = 1$.
Note that these $z^{(i)}, p^{(i)}, \eta^{(i)}$ values are determined for a particular
$(k,y)$, and hence we can relabel them as $z_k^{(i)}(y)$, $p_k^{(i)}(y)$, and $\eta_k^{(i)}(y)$
for $i \in \{1, 2\}$. 

Now define the following stationary randomized policy:  For each product $k\in\{1, \ldots, K\}$, 
if $Y(t) = y$, 
independently choose $Z_k^*(t) = z_k^{(1)}(y)$ and $P_k^*(t) = p_k^{(1)}(y)$ with probability 
$\eta_k^{(1)}(y)$, and else choose $Z_k^*(t) = z_k^{(2)}(y)$ and $P_k^*(t) = p_k^{(2)}(y)$. 
It follows that for a given value of $y$,  
this policy uses at most two different prices for each product.\footnote{Further, given the observed
value of $Y(t) = y$, this policy makes pricing
decisions independently for each product $k$.} Further, we have: 
 \begin{eqnarray*}
 \hat{r}_k(y) &\leq& \expect{Z_k^*(t)(P_k^*(t)-\alpha_k)F_k(P_k^*(t), y)\left|\right. Y(t) = y}  \\
 \hat{d}_{k}(y) &=& \expect{Z_k^*(t)F_k(P_k^*(t), y)\left|\right. Y(t) = y} 
 \end{eqnarray*}
Summing these conditional expectations over $k \in \{1, \ldots, K\}$ and $y \in \script{Y}$
and using (\ref{eq:baz1})-(\ref{eq:baz2}) yields the result. 
\end{proof}

\section*{Appendix C --- Proof of Lemma \ref{lemma:sample_path_Tdrift}}
\begin{proof} (Lemma \ref{lemma:sample_path_Tdrift}) 
Using the queueing dynamic equation (\ref{eq:dynamics1}) (which holds because $Q_m(t) \geq \mu_{m,max}$
for all $t$),  it is easy to show: 
\begin{eqnarray*}
\frac{1}{2}\big(Q_m(\tau+1)-\theta_m\big)^2-\frac{1}{2}\big(Q_m(\tau)-\theta_m\big)^2 \\
=  \frac{1}{2}(A_m(t) - \mu_m(t))^2 +(Q_m(\tau)-\theta_m)\big[A_m(\tau)-\mu_m(\tau)\big],
\end{eqnarray*}
Now summing over $m\in\{1, \ldots, M\}$ and adding to both sides the term $-V\phi(\tau)$, we have:
\begin{eqnarray*}
\tilde{\Delta}_1(\tau) -V \phi(\tau)\leq B-V \phi(\tau) + \sum_{m=1}^M(Q_m(\tau)-\theta_m)\big[A_m(\tau)-\mu_m(\tau)\big].
\end{eqnarray*}
where $B$ is defined in (\ref{eq:B}), and $\tilde{\Delta}_1(\tau)\defequiv 
\frac{1}{2}\sum_{m=1}^M\big[\big(Q_m(\tau+1)-\theta_m\big)^2- \big(Q_m(\tau)-\theta_m\big)^2\big]$ is the $1$-step sample path drift of the Lyapunov function at time $\tau$. 
Now for any $t_0\leq\tau\leq t_0+T-1$, we can take expectations over the above equation conditioning on $\hat{H}_T(t_0)$ to get:
\begin{eqnarray}
\hspace{-.1in}&&\expect{\tilde{\Delta}_1(\tau)|\hat{H}_T(t_0)} -V \expect{\phi(\tau)\left.|\right.\hat{H}_T(t_0)}\leq B-V \expect{\phi(\tau)\left.|\right.\hat{H}_T(t_0)} \label{eq:1step_drift_ep}\\
\hspace{-.1in}&&\qquad\qquad\qquad\qquad\qquad+ \sum_{m=1}^M\expect{(Q_m(\tau)-\theta_m)\big[A_m(\tau)-\mu_m(\tau)\big]\left.|\right.\hat{H}_T(t_0)}.\nonumber
\end{eqnarray}
However, using iterated expectations in the last term as in (\ref{eq:foo}), we see that:
\begin{eqnarray*}
\hspace{-.1in}&&\quad\sum_{m=1}^M\expect{(Q_m(\tau)-\theta_m)\big[A_m(\tau)-\mu_m(\tau)\big]\left.|\right.\hat{H}_T(t_0)} \\
\hspace{-.1in}&&= \sum_{m=1}^M\expect{(Q_m(\tau)-\theta_m)\expect{A_m(\tau)-\mu_m(\tau) \left.|\right.\bv{Q}(\tau), \bv{P}(\tau), \bv{Z}(\tau), \hat{H}_T(t_0)}\left.|\right.\hat{H}_T(t_0)}\\
\hspace{-.1in}&&= \sum_{m=1}^M\expect{(Q_m(\tau)-\theta_m)\big[A_m(\tau)- \sum_{k=1}^K\beta_{mk}Z_k(\tau)F_k(P_k(\tau), Y(\tau))\big] \left.|\right.\hat{H}_T(t_0)}
\end{eqnarray*}
Plugging this back into (\ref{eq:1step_drift_ep}), we get:
\begin{eqnarray}
\hspace{-.1in}&&\expect{\tilde{\Delta}_1(\tau)|\hat{H}_T(t_0)} -V \expect{\phi(\tau)\left.|\right.\hat{H}_T(t_0)}\leq B-V \expect{\phi(\tau)\left.|\right.\hat{H}_T(t_0)} \label{eq:1step_drift_ep2}\\
\hspace{-.1in}&& + \sum_{m=1}^M\expect{(Q_m(\tau)-\theta_m)\big[A_m(\tau)-\sum_{k=1}^K\beta_{mk}Z_k(\tau)F_k(P_k(\tau), Y(\tau))\big]\left.|\right.\hat{H}_T(t_0)}.\nonumber
\end{eqnarray}
Now since, given the $\bv{Q}(\tau)$ values on slot $\tau$,  
the JPP algorithm minimizes the right hand side of the above equation at time $\tau$, we indeed have:
\begin{eqnarray}
\hspace{-.1in}&&\expect{\tilde{\Delta}_1(\tau)|\hat{H}_T(t_0)} -V \expect{\phi(\tau)\left.|\right.\hat{H}_T(t_0)}\leq B-V \expect{\phi^*(\tau)\left.|\right.\hat{H}_T(t_0)} \label{eq:1step_drift_ep3}\\
\hspace{-.1in}&& + \sum_{m=1}^M\expect{(Q_m(\tau)-\theta_m)\big[A^*_m(\tau)-\sum_{k=1}^K\beta_{mk}Z^*_k(\tau)F_k(P^*_k(\tau), Y(\tau))\big]\left.|\right.\hat{H}_T(t_0)},\nonumber
\end{eqnarray}
where $\phi^*(\tau), A^*_m(\tau)$ $Z_k^*(\tau)$ and $P^*_k(\tau)$ are variables generated by any other policies. Summing (\ref{eq:1step_drift_ep3}) from $\tau=t_0$ to $\tau=t_0+T-1$, we thus have:
\begin{eqnarray}
\hspace{-.1in}&&\hat{\Delta}_T(t_0) -\sum_{\tau=t_0}^{t_0+T-1}V \expect{\phi(\tau)\left.|\right.\hat{H}_T(t_0)}\leq TB-V\sum_{\tau=t_0}^{t_0+T-1} \expect{\phi^*(\tau)\left.|\right.\hat{H}_T(t_0)} \label{eq:1step_drift_ep4}\\
\hspace{-.1in}&& + \sum_{m=1}^M\sum_{\tau=t_0}^{t_0+T-1}\expect{(Q_m(\tau)-\theta_m)\big[A^*_m(\tau)-\sum_{k=1}^K\beta_{mk}Z^*_k(\tau)F_k(P^*_k(\tau), Y(\tau))\big]\left.|\right.\hat{H}_T(t_0)}.\nonumber
\end{eqnarray}
Now using the fact that for any $t$, $|Q_m(t+\tau)-Q_m(t)|\leq \tau \max[A_{m, max}, \mu_{m,max}]$, we get:
\begin{eqnarray*}
\hspace{-.1in}&&\sum_{m=1}^M\sum_{\tau=t_0}^{t_0+T-1}(Q_m(\tau)-\theta_m)\big[A^*_m(\tau)-\sum_{k=1}^K\beta_{mk}Z^*_k(\tau)F_k(P^*_k(\tau), Y(\tau))\big]\\
\hspace{-.1in}&&\leq B' 
+\sum_{m=1}^M\sum_{\tau=t_0}^{t_0+T-1}(Q_m(t_0)-\theta_m)\big[A^*_m(\tau)-\sum_{k=1}^K\beta_{mk}Z^*_k(\tau)F_k(P^*_k(\tau), Y(\tau))\big]
\end{eqnarray*}
where: 
 \[ B' \defequiv \frac{T(T-1)}{2}\sum_{m=1}^M\max[A_{m,max}^2, \mu_{m,max}^2] = T(T-1)B\]
 where $B$ is defined in (\ref{eq:B}). 
Plugging this into (\ref{eq:1step_drift_ep4}) and using the fact that conditioning on $\hat{H}_T(t_0)$, $\bv{Q}(t_0)$ is a constant, we  get:
\begin{eqnarray*}
\hspace{-.2in}&&\hat{\Delta}_T(t_0)-V \sum_{\tau=t_0}^{t_0+T-1}\expect{\phi(\tau)\left.|\right. \hat{H}_T(t_0)}\leq 
TB + B' -V \sum_{\tau=t_0}^{t_0+T-1}\expect{\phi^*(\tau)\left.|\right. \hat{H}_T(t_0)} \\
\hspace{-.2in}&&  +\sum_{m=1}^M\big(Q_m(t_0)-\theta_m\big)\sum_{\tau=t_0}^{t_0+T-1}\expect{A^*_m(\tau)-\sum_{k=1}^K\beta_{mk}Z^*_k(\tau)F_k(P^*_k(\tau), Y(\tau))\left.|\right. \hat{H}_T(t_0)},\nonumber
\end{eqnarray*}
Noting that $TB + B' = T^2B$ proves Lemma  \ref{lemma:sample_path_Tdrift}. 
\end{proof}

\bibliographystyle{../../../../acm-latex/acmtrans}
\bibliography{../../../../latex-mit/bibliography/refs}
\end{document}